\documentclass[10pt]{amsart}
\usepackage{graphicx,enumitem,dsfont}
\usepackage[colorlinks]{hyperref}
\usepackage{subfigure}
\usepackage{pdfsync}
\usepackage{upref}

\newcommand{\R}{\mathbb{R}}
\newcommand{\Z}{\mathbb{Z}}

\newcommand{\seq}[1]{\left\{#1\right\}}

\newcommand{\Dx}{{\Delta x}}
\newcommand{\Dt}{{\Delta t}}
\newcommand{\norm}[1]{\left\|#1\right\|}
\newcommand{\abs}[1]{\left|#1\right|}

\newcommand{\Dp}{D_{+}}
\newcommand{\Dmx}{D_-}
\newcommand{\Dpt}{D_+^t}
\newcommand{\Dmt}{D_-^t}
\newcommand{\weakto}{\rightharpoonup}
\newcommand{\weakstarto}{\overset{\star}{\weakto}}
\newcommand{\eps}{\varepsilon}
\newcommand{\CL}{\mathcal{L}}
\newcommand{\N}{\mathbb{N}}
\newcommand{\CMloc}{\mathcal{M}_{\mathrm{loc}}}
\newcommand{\Dm}{D_-}
\newcommand{\jint}{\int_{x_{j-1/2}}^{x_{j+1/2}}\!\!}
\newcommand{\tint}{\int_{t_n}^{t_{n+1}}\!\!}
\newcommand{\loc}{{\mathrm{loc}}}

\renewcommand{\u}[1]{u^{(#1)}}
\newcommand{\Div}{\mathrm{div}}
\newcommand{\Curl}{\mathrm{curl}}
\newcommand{\st}{\;\bigm|\;}

\newtheorem{definition}{Definition}[section]

\newtheorem{theorem}{Theorem}[section]
\newtheorem{lemma}{Lemma}[section]

\newtheorem{remark}{Remark}[section]
\theoremstyle{definition} 

\newtheorem*{maintheorem*}{Main Theorem}

\allowdisplaybreaks

\numberwithin{equation}{section}
\numberwithin{figure}{section}
\numberwithin{table}{section}

\newcounter{asnr}

\newenvironment{Assumptions} %
{\ifnum\value{asnr}=0 \stepcounter{asnr} 
  \begin{enumerate}[label=\textbf{A}.\arabic{enumi}]
    \else
    \begin{enumerate}[label=\textbf{A}.\arabic{enumi},resume] \fi}
{\end{enumerate}}

\newcounter{defnr}

\newenvironment{Definitions} %
{\ifnum\value{defnr}=0 \stepcounter{defnr} 
  \begin{enumerate}[label=\textbf{D}.\arabic{enumi}]
    \else
    \begin{enumerate}[label=\textbf{D}.\arabic{enumi},resume] \fi}
{\end{enumerate}}

\title[Keyfitz-Kranzer]{Finite difference schemes for 
 the symmetric Keyfitz-Kranzer system} 

\author[U. Koley]{U. Koley} \address[Ujjwal Koley] {\newline  
   Institut f\"{u}r Mathematik,  \newline
   Julius-Maximilians-Universit\"{a}t W\"{u}rzburg,   
\newline Campus Hubland Nord, Emil-Fischer-Strasse 30, \newline 97074,
W\"{u}rzburg, Germany.} 
\email[]{toujjwal@gmail.com}

\author[N. H. Risebro]{N. H. Risebro} \address[Nils Henrik
Risebro]{\newline Centre of Mathematics for Applications (CMA) 
  \newline University of Oslo\newline P.O. Box 1053, Blindern\newline
 N--0316 Oslo, Norway} \email[]{nilshr@math.uio.no}

\thanks{This paper was written when NHR  was a quest of the
  Seminar f\"{u}r Angewandte Mathematik, ETH, Zürich. This institution
  is thanked for its hospitality.  UK was supported in part by a
  Humboldt Research Fellowship through  the Alexander von Humboldt
  Foundation.
  The anonymous referee is warmly thanked for the thorough reading of
  the manuscript, and for the many useful remarks and suggestions.}


\date{\today}

\begin{document}

\begin{abstract}
  We are concerned with the convergence of numerical schemes for the
  initial value problem associated to the Keyfitz-Kranzer system of
  equations. This system is a toy model for several important models
  such as in elasticity theory, magnetohydrodynamics, and enhanced oil
  recovery. In this paper we prove the convergence of three difference
  schemes. Two of these schemes is shown to converge to the unique
  entropy solution.  Finally, the convergence is illustratred by
  several examples.
\end{abstract}

\maketitle

\tableofcontents

\section {Introduction}
In this paper, we consider difference methods for the Cauchy problem
for the $n \times n$ symmetric system of Keyfitz-Kranzer type
\begin{equation}
  \label{eq:system}
  \begin{cases}
    u_t + \left(u\phi(\abs{u})\right)_x = 0, &\ \ x \in \Omega = \R \times (0,T),\\
    u(x,0)=u_0(x), &\ \ x \in \R,
  \end{cases}
\end{equation}
where $T>0$ is fixed, $u=\left(\u{1},\ldots,\u{n}\right): \R \times
[0,T) \rightarrow \R^n$ is the unknown vector map with
$\abs{u}=\sqrt{{\u{1}}^2+\cdots+{\u{n}}^2}$,
$u_0=\left(u_0^{(1)},\ldots,u_0^{(n)}\right)$ the initial data, and $
\phi: \R \rightarrow \R$ is given (sufficiently smooth) scalar
function (see Section ~\ref{sec:math} for the complete list of
assumptions). Systems of this type was first considered in
\cite{KeyfitzKranzer, LiuWang} and later on by several other authors
\cite{DeLelis}, as a prototypical example of a non-strictly hyperbolic
system. This type of system is a model system for some phenomena in
magnetohydrodynamics, elasticity theory and enhanced
oil-recovery. This system also has similarities to a model of
chromatography \cite{Kroma} and to a model describing polymer flooding
in porous media \cite{TveitoWinter}. Note that since $\phi$ is a
function of $\abs{u}$, we call \eqref{eq:system} a symmetric
Keyfitz-Kranzer system. A non-symmetric version of the Keyfitz-Kranzer
system reads
\begin{equation}
  \label{eq:system_non}
  \begin{cases}
    u_t + \left(u\,\phi(u, w_1, w_2, \ldots, w_n)\right)_x = 0, &\\
    (u\,w_i)_t + \left(u\, w_i\,\phi(u, w_1, w_2, \ldots,
      w_n)\right)_x = 0, &\ \ i = 1,2, \ldots, n.
  \end{cases}
\end{equation}
Existence of global bounded weak solutions to \eqref{eq:system_non}
has been studied by Lu \cite{Lu2} for a specific choice of $\phi$.

For the flux function $F(u)= u \phi(\abs{u})$, a straightforward
calculation shows $B(u)= dF(u)$ is the matrix with entries
\begin{align*}
  B_{i,j} (u) = \phi(\abs{u}) \delta_{i,j} + \phi'(\abs{u}) \frac{u_i
    u_j}{\abs{u}}, \quad i,j = 1,2, \cdots, n,
\end{align*}
where $\delta_{i,j}$ is the Kronecker delta, given by
\begin{align*}
  \delta_{i,j} =
  \begin{cases}
    1, \ \ & i=j \\
    0, \ \ & i \neq j.
  \end{cases}
\end{align*}
The matrix $B(u)$ is symmetric, therefore the system \eqref{eq:system}
is hyperbolic, that is, all the eigenvalues of $B(u)$ are real and the
corresponding collection of eigenvectors is complete. It is easy to
see that the first eigenvalue of $B(u)$ is $\lambda_1 = \phi(\abs{u})
+ \phi'(\abs{u}) \abs{u}$ and other $n-1$ eigenvalues are $\lambda_i =
\phi(\abs{u})$, $i=1,2,\cdots, n-1$. The presence of repeated
eigenvalues shows that the system \eqref{eq:system} is not strictly
hyperbolic.

Due to the nonlinearity, discontinuities in the solution may appear
independently of the smoothness of the initial data and weak solution
must be sought.  A weak solution is defined as follows:
\begin{definition}
  We say $u(x,t)$ a weak solution to \eqref{eq:system} if
  \begin{Definitions}
  \item $u(x,t) \in L^{\infty}(\R \times \R_{+})$.
  \item For all test functions $\psi\in C^\infty_0(\R\times
    [0,\infty))$
    \begin{equation}
      \label{eq:weaksys_1}
      \iint_{\R\times\R^+} \!\! u\psi_t + u\phi(\abs{u})\psi_x \,dxdt +
      \int_\R u_0 \psi(x,0)\,dx = 0,
    \end{equation}
  \end{Definitions}
\end{definition}
It is well known that weak solutions may be discontinuous and they are
not uniquely determined by their initial data. Consequently, an
entropy condition must be imposed to single out the physically correct
solution. Therefore the Cauchy problem is viewed in the framework of
entropy solutions. For \eqref{eq:system}, an entropy formulation was
first introduced by Freist\"{u}hler \cite{Freistuler1,Freistuhler},
and independently, by Panov \cite{Panov2}. An entropy solution to
\eqref{eq:system} is defined as follows:
\begin{definition}
  \label{def:entropy}
  A bounded measurable function $u(x,t)$ is called an entropy solution
  to \eqref{eq:system} if
  \begin{Definitions}
  \item For all test functions $\psi\in C^\infty_0(\R\times
    [0,\infty))$
    \begin{equation}
      \label{eq:weaksys}
      \iint_{\R\times\R^+} \!\! u\psi_t + u\phi(\abs{u})\psi_x \,dxdt +
      \int_\R u_0 \psi(x,0)\,dx = 0,
    \end{equation}
  \item $r=\abs{u}$ is an entropy solution (in the sense of
    Kru\v{z}kov \cite{kruzkov}) of the scalar conservation law
    \begin{equation}\label{eq:scalar}
      \begin{cases}
        r_t + \left(r\phi(r)\right)_x = 0, \ \ &t>0,\\
        r(x,0)=\abs{u_0(x)}.
      \end{cases}
    \end{equation}
  \end{Definitions}
\end{definition}
Regarding the existence, uniqueness of solutions and continuous
dependence of solutions on the initial data we have the following
result
\begin{theorem}
  \label{thm:FreistuhlerPanov}
  The system \eqref{eq:system} has the following properties:
  \begin{itemize}
  \item [(E)] The system has a solution for any $u_0\in L^\infty(\R)$.
  \item [(U)] For such $u_0$, there is precisely one solution $u$ with
    the property that $r=\abs{u}$ satisfies the scalar conservation
    laws \eqref{eq:scalar} and Kru\v{z}kov's entropy criterion.
  \item [(S)] This solution $u$ depends $L^1_\loc(\R)$ continuously on
    the initial data $u_0$.
  \end{itemize}
\end{theorem}
This theorem was first proved in \cite{Freistuhler} by using the
famous equivalence result of Wagner \cite{Wagner}. The key idea behind
this proof is to view the system \eqref{eq:system} as an extended
system, consisting of \eqref{eq:system} and an additional conservation
law satisfied by $r$ \eqref{eq:scalar}, with Wagner's transformation
theory. On the other hand, in \cite{Panov2}, Panov gave a ``direct''
proof of both existence and uniqueness. The existence was proved by
showing the convergence of the singularly perturbed problems
\begin{equation*}
  u^\eps_t + \left(u^\eps \phi(\abs{u^\eps})\right)_x = \eps u^\eps_{xx},
\end{equation*}
to an entropy solution as $\eps\to 0$. The idea behind the existence
proof was first to show the existence of a \emph {measure-valued}
solution $\nu_{(t,x)}$ of the Cauchy problem \eqref{eq:system}. Then
he showed that indeed $\nu_{(t,x)}$ is regular: $\nu_{(t,x)}(u)=
\delta(u -u(t,x)), u(t,x) \in L^{\infty}(\R \times \R^{+}, \R^n)$ and
consequently this gives existence of a solution to \eqref{eq:system}.

In view of the analytic properties of the solutions of
\eqref{eq:system}, several different methods for computing the
solution suggest themselves. Foremost among these methods is Glimm's
scheme \cite{Glimm}. Regarding other numerical methods, it is tempting
to use the equation satisfied by $r$, and view $r$ as an independent
variable. Defining $v\in S^{n-1}$ by $vr = u$, we formally have that
\begin{align}
  r_t + (r\phi(r))_x &= 0\label{eq:req}\\
  (rv)_t + (r\phi(r) v)_x &=0\label{eq:veq1} \intertext{or} v_t +
  \phi(r) v_x &=0.\label{eq:veq2}
\end{align}
As a strategy, one can then solve \eqref{eq:req} first, and then
either \eqref{eq:veq1} or \eqref{eq:veq2}. These should then hold
subject to the constraint $\abs{v}=1$. Without this constraint,
\eqref{eq:req}--\eqref{eq:veq1} is a ``triangular'' system of
conservation laws, see \cite{Triang_Cocliteetal}. Using any monotone
scheme for \eqref{eq:req} and \eqref{eq:veq1} will ensure the strong
convergence of the approximate solutions to \eqref{eq:req} and the
weak-star convergence of the approximate solutions to
\eqref{eq:veq1}. This approach was used in
\cite{FreistuhlerPittman}. To show that $u=rv$ is an entropy solution
to \eqref{eq:system}, one must show (for the approximations) that
$\abs{v}=1$ in the limit if $\abs{v_0}=1$.  In this paper, for the
semi-discrete scheme, we discretize \eqref{eq:system} in space and
show the convergence of approximate solution to a weak solution of
\eqref{eq:system}. But we are unable to extend our analysis to the
fully discrete scheme based on discretizing \eqref{eq:system}.  To
overcome this difficulty, we propose another scheme based on
discretizing \eqref{eq:req}--\eqref{eq:veq2} and prove the convergence
of approximate solution to unique entropy solution of
\eqref{eq:system}.

The present paper can be divided into four parts:
\begin{enumerate}
\item In Section ~\ref{sec:math}, we present the mathematical framework
  used in this paper. In particular, we used a compensated compactness
  result in the spirit of Tartar \cite{Tartar} but the proof is based
  on div-curl lemma and does not rely on the Young measure.

\item In section ~\ref{sec:semi}, we propose an upwind semi-discrete
  finite difference scheme and prove the convergence of the
  approximate solution to the weak solution of \eqref{eq:system}. The
  idea behind this proof is to prove first the strong convergence of
  approximate solution $r_{\Dx} = \abs{u_{\Dx}}$ using the compensated
  compactness technique \cite{Tartar,Chen}. Then prove a $B.V.$ estimate
  for $\tau=u/(u,e)$, where $e$ is a unit vector in $\R^n$.
  Then Kolmogorov's theorem, combined with
  the strong convergence of $r_{\Dx}$, gives the strong convergence of
  approximate solution $u_{\Dx}$.

\item In section ~\ref{sec:fully}, for a fully discrete scheme, we are
  only able to conclude that $u$ is only a distributional solution of
  \begin{align*}
    u_t + (u \, \phi(r))_x =0,
  \end{align*}
  for some $r$ such that $\abs{u} \le r$.  We propose another fully
  discrete scheme relying on explicit decoupling of the variables $r$
  and $v$ expressed by the ``nonconservative'' formulation
  \eqref{eq:req}--\eqref{eq:veq2}
  \begin{equation*}
    \begin{cases}
      r_t + (r\phi(r))_x &= 0, \\
      v_t + \phi(r) \, v_x &=0,
    \end{cases}
  \end{equation*}
  with $r(0) = \abs{u(0)}$.  It is not difficult to show the
  convergence of $r_{\Dx}$ to $r$, $r$ being the unique entropy
  solution of \eqref{eq:req}, and the strong convergence of $v_{\Dx}$.
  In order to conclude that $u= r\,v$ is the unique entropy solution
  of \eqref{eq:system}, one has to show $\abs{v(x,t)} =1$ and this has
  been achieved in this paper using Wagner transformation
  \cite{Wagner} (see Section ~\ref{sec:math} for more details).

\item Finally, in Section ~\ref{sec:numerical}, we test our numerical
  schemes and provide some numerical results.

\end{enumerate}

\section{Mathematical Framework}
\label{sec:math}
In this section we present some mathematical tools that we shall use
in the analysis. To start with the basic assumptions on the initial
data and the funtion $\phi(r)$, we assume that $\phi$ is a twice
differentiable function $\phi:[0,\infty)\to [0,\infty)$ so that
\begin{Assumptions}
\item $\phi(0)=0$, $\phi(r)>0$ and $ \phi'(r)\ge 0$ for all relevant
  $r$;\label{def:w1}
\item $\phi(r), \phi'(r)$ and $\phi''(r)$ are bounded for all relevant
  $r$;\label{def:w2}
\item $ \mathrm{meas}\seq{ r \,\Bigm|\, 2\phi'(r) + r\phi''(r) = 0} =
  0$;\label{def:w3}
\item $ \abs{u_0} \in L^1(\R)\cap L^{\infty}(\R)$ and $ \abs{u_0} \in \mathcal{B}(K)$ for any constant $K$ in $\R$, where
\begin{align*}
\mathcal{B}(K) & := \lbrace  f \st \inf_{x \in \mathcal{A}(f,K)} f \ge C_K \rbrace,\, \text{and} 
\end{align*}
\begin{align*}
\mathcal{A}(f, K) &:= 
& \seq{
  \begin{gathered}
    x \in (-\infty, K] \st \exists \,\eps >0\,\\
    \text{with}\, f(y) > \liminf_{z \rightarrow x} f(z), \, \text{for
      a.e} \,\, y \in (x-\eps, x)
  \end{gathered}
}.
\end{align*}
Here $C_K$ is a positive constant depending on $K$;\label{def:w4}
\item $u_0 \in \Gamma_\delta$, where $\Gamma_\delta$ is the cone
  \begin{equation*}
    \Gamma_\delta :=\seq{ u\in \R^n \st \delta \abs{u}\le (e,u) }
  \end{equation*}
  for some fixed unit vector (which we without loss of generality
  choose as $e=(1,\ldots,1)/\sqrt{n}$), and $\delta$ is a fixed number
  in the interval $(\sqrt{(n-1)/n},1)$.\label{def:w5}
\end{Assumptions}
Next, we recapitulate the results we shall use from the compensated
compactness method due to Murat and Tartar \cite{Murat,Tartar}. For a
nice overview of applications of the compensated compactness method to
hyperbolic conservation laws, we refer to Chen \cite{Chen}.  Let
$\mathcal{M}(\R)$ denote the space of bounded Radon measures on $\R$
and
\begin{align*}
  C_0(\R) = \seq{ \psi \in C(\R) \st \lim_{\abs{x} \rightarrow \infty}
  \psi(x) =0 }.
\end{align*}
If $\mu \in \mathcal{M}(\R)$, then
\begin{align*}
  \langle \mu, \psi \rangle = \int_{\R} \psi \,d \mu, \quad \text{for
    all} \quad \psi \in C_0(\R).
\end{align*}
Recall that $\mu \in \mathcal{M}(\R)$ if and only if $ \abs{\langle
  \mu, \psi \rangle} \le C \norm{\psi}_{L^{\infty}(\R)} $ for all
$\psi \in C_0(\R)$. We define
\begin{align*}
  \norm{\mu}_{\mathcal{M}(\R)} = \sup{\lbrace \abs{\langle \mu, \psi
      \rangle}: \psi \in C_0(\R), \norm{\psi}_{L^{\infty}(\R)} \le 1
    \rbrace}.
\end{align*}
The space $\left( \mathcal{M}(\R), \norm{\cdot}_{\mathcal{M}(\R)}
\right)$ is a Banach space and it is isometrically isomorphic to the
dual space of $\left(C_0(\R), \norm{\cdot}_{L^{\infty}(\R)} \right)$,
while we define the space of probablity measures
\begin{align*}
  \text{Prob}(\R) = \lbrace \mu \in \mathcal{M}(\R): \mu \, \text{is
    nonnegative and} \, \norm{\mu}_{\mathcal{M}(\R)}=1 \rbrace.
\end{align*}

Before we state the compensated compactness theorem, we recall the
celebrated div-curl lemma.
\begin{lemma}[div-curl lemma]
  \label{lem:div}
  Let $\Omega$ be a bounded open subset of $\R^2$. With $\eps >0$
  denoting a parameter taking its value in a sequence which tends to
  zero, suppose
  \begin{align*}
    & D^{\eps} \weakto D \, \, \text{in} \,\, (L^2(\Omega))^2, \qquad E^{\eps} \weakto E \,\, \text{in} \, \, (L^2(\Omega))^2, \\
    &{\lbrace \Div \,D^{\eps} \rbrace}_{\eps>0} \, \, \text{lies in a compact subset of} \, \, H^{-1}_{\mathrm{loc}}(\Omega), \\
    &{\lbrace \Curl \,E^{\eps} \rbrace}_{\eps>0} \, \, \text{lies in a
      compact subset of} \, \, H^{-1}_{\mathrm{loc}}(\Omega).
  \end{align*}
  Then along a subsequence
  \begin{align*}
    D^{\eps} \cdot E^{\eps} \rightarrow D \cdot E \, \, \text{in} \,\,
    \mathcal{D}'(\Omega).
  \end{align*}
\end{lemma}
We shall use the following compensated compactness result.
\begin{theorem}
  \label{thm:compcomp} Let $\Omega\subset \R\times \R^+$ be a bounded
  open set, and assume that $\seq{u^\eps}$ is a sequence of uniformly
  bounded functions such that $\abs{u^\eps}\le M$ for all $\eps$. Also
  assume that $f:[-M,M] \to \R$ is a twice differentiable
  function. Let $u^\eps \weakstarto u$ and $f(u^\eps)\weakstarto v$,
  and set
  \begin{equation}
    \begin{aligned}
      \left(\eta_1(s),q_1(s)\right) &= \left(s-k,f(s)-f(k)\right),\\
      \left(\eta_2(s),q_2(s)\right) &= \left(f(s)-f(k), \int_k^s
        (f'(\theta))^2\,d\theta \right),
    \end{aligned}\label{eq:entropies}
  \end{equation}
  where $k$ is an arbitrary constant. If
  \begin{equation*}
    \eta_i(u^\eps)_t + q_i(u^\eps)_x \ \text{ is in a compact set of
      $H^{-1}_{\mathrm{loc}}(\Omega)$ for $i=1$, $2$,}
  \end{equation*}
  then
  \begin{enumerate}
  \item $v=f(u)$, a.e.~$(x,t)$,
  \item $u^\eps \to u$, a.e.~$(x,t)$ if $\mathrm{meas}\seq{u\,|\,
      f''(u)=0}=0$.
  \end{enumerate}
\end{theorem}
For a proof of this theorem, see the monograph of Lu \cite{Lu}. A
feature of the compensated compactness result above is that it avoids
the use of the Young measure by following an approach developed by
Chen and Lu \cite{Lu,Chen} for the standard scalar conservation
law. This is preferable as the fundamental theorem of Young measures
applies most easily to functions that are continuous in all variables.

The following compactness interpolation result (known as Murat's lemma
\cite{Murat}) is useful in obtaining the $H^{-1}_{\loc}$ compactness
needed in Theorem ~\ref{thm:compcomp}.
\begin{lemma}
  \label{lem:Murat}
  Let $\Omega$ be a bounded open subset of $\R^2$.  Suppose that the
  sequence $\seq{\CL_\eps}_{\eps>0}$ of distributions is bounded in
  $W^{-1,\infty}(\Omega)$.  Suppose also that
  $$
  \CL_\eps=\CL_{1,\eps} + \CL_{2,\eps},
  $$
  where $\seq{\CL_{1,\eps}}_{\eps>0}$ is in a compact subset of
  $H^{-1}(\Omega)$ and $\seq{\CL_{2,\eps}}_{\eps>0}$ is in a bounded
  subset of $\CMloc(\Omega)$.  Then $\seq{\CL_\eps}_{\eps>0}$ is in a
  compact subset of $H^{-1}_{\mathrm{loc}}(\Omega)$.
\end{lemma}
Next, we shall need Kolmogorov's compactness lemma.
\begin{lemma}[$L^1_{\loc}$ compactness, see \cite{Holden}]
  \label{lem:kolmogorov}
  Let $u^{\eps} : \R \times [0,\infty) \rightarrow \R$ be a family of
  functions such that for each positive $T$,
  \begin{align*}
    \abs{u^{\eps} (x,t)} \le C_T, \,\, (x,t) \in \R \times [0,T]
  \end{align*}
  for a constant $C_T$ independent of $\eps$. Assume in addition that
  for all compact $B \subset \R$ and for $t \in [0,T]$
  \begin{align*}
    \sup_{\abs{\xi} \le \abs{\rho}} \int_{B} \abs{u^{\eps}(x + \xi, t)
      - u^{\eps} (x,t)} \,dx \le \nu_{B,T}(\abs{\rho}),
  \end{align*}
  for a modulus of continuity $\nu_{B,T}$. Furthermore, assume for $s$ and
  $t$ in $[0,T]$ that
  \begin{align*}
    \int_{B} \abs{u^{\eps}(x , t) - u^{\eps} (x,s)} \,dx \le
    \omega_{B,T}(\abs{t-s}) \,\, \text{as} \,\, \eps \downarrow 0,
  \end{align*}
  for some modulus of continuity $\omega_{B,T}$. Then there exists a
  sequence $ \seq{\eps_j}$ such that for each $t \in [0,T]$
  the function $\lbrace u^{\eps_j}(t) \rbrace$ converges to a function
  $u(t)$ in $L^{1}_{\loc}(\R)$. The convergence is in $C([0,T];
  L^1_{\loc}(\R))$.
\end{lemma}
Finally, we state the following result related to Wagnar
transformation theory.
\begin{lemma}[Wagner Transformation, see \cite{Wagner,Freistuler1}]
  \label{lem:wagnar}
  For any $n \in \N$, there is a one-to-one correspondence between
  (equivalence classes of) bounded Lebesgue measurable solutions $(r,
  rv): \R^2_{+} \rightarrow [0, \infty) \times \R^m$ to the system
  \eqref{eq:req}--\eqref{eq:veq1} which satisfy
  \begin{align*}
    \int_{-\infty}^0 r(x,t)\,dx = \int_0^{\infty} r(x,t)\,dx = \infty
  \end{align*}
  and (equivalence classes of) weak solutions $(\tau, \tilde{v})$ to
  the system
  \begin{equation}
    \begin{aligned}
      \tau_t - \left(\phi(1/\tilde{\tau}) \right)_y & = 0, \\
      \tilde{v}_t & = 0,
    \end{aligned}
    \label{eq:newsys}
  \end{equation}
  in which $\tau$ is a Radon measure in $\R^2_{+}$ which dominates
  Lebesgue (outer) measure $\lambda_2$ (i.e., $\tau \ge k \lambda_2$
  for some $k > 0$), $\tilde{\tau}$ is the density of the absolutely
  continuous part of $\tau$ with respect to $\lambda_2$, and
  $\tilde{v}: \R^2_{+} \rightarrow \R^m$ is bounded and Lebesgue
  measurable. This correspondence is established through
  transformations $T:(x,t) \rightarrow (y(x,t),t)$ defined relative to
  any bounded measurable solutions to \eqref{eq:req} by
  \begin{equation}
    \begin{aligned}
      \frac{\partial y}{\partial x} (x,t) = r(x,t), \qquad
      \frac{\partial y}{\partial t} (x,t) = - \phi(r(x,t))\,r(x,t),
      \qquad y(0,0)=0,
    \end{aligned}
    \label{eq:yeq}
  \end{equation}
  namely setting
  \begin{equation}
    \begin{aligned}
      \tau & = \lambda_2 \circ T^{-1}, \\
      \tilde{v} & = v \circ T^{-1}.
    \end{aligned}
    \label{eq:transformation}
  \end{equation}
\end{lemma}
Observe that using \eqref{eq:newsys}, \eqref{eq:yeq} and
\eqref{eq:transformation}, it is easy to conclude that $\abs{v} =1$.

\section{A semi-discrete finite difference scheme}
\label{sec:semi}
We start by introducing some notation needed to define the
semi-discrete finite difference schemes. Throughout this paper we
reserve $\Dx$ to denote a small positive number that represent the
spatial discretization parameter of the numerical schemes. Given
$\Dx>0$, we set $x_j=j\Dx$ for $j\in \Z$ and for any function $u =
u(x)$ admitting pointvalues we write $u_j = u(x_j)$. Furthermore, let
us introduce the spatial grid cells
\begin{align*}
  I_j = [x_{j-1/2}, x_{j+1/2}),
\end{align*}
where $x_{j\pm1/2} = x_j \pm \Dx/2$. Let $D_{\pm}$ denote the discrete
forward and backward differences, i.e.,
\begin{equation*}
  D_{\pm}u_j = \mp\frac{u_j - u_{j\pm 1}}{\Dx}. 
\end{equation*}
The discrete Leibnitz rule is given by
\begin{align*}
  D_{\pm} (u_j v_j) = u_j D_{\pm} v_j + v_{j \pm 1} D_{\pm} u_j
\end{align*}
Furthermore, for any $C^2$ function $f$, using the Taylor expansion on
the sequence $f(u_j)$ we obtain
\begin{align*}
  D_{\pm} f(u_j) = f'(u_j) D_{\pm} u_j \pm \frac{\Dx}{2} f''(\xi_{j
    \pm \frac{1}{2}}) (D_{\pm} u_j)^2,
\end{align*}
for some $\xi_{j\pm \frac{1}{2}}$ between $u_{j \pm 1}$ and $u_j$. We
will make frequent use of this, which states that a discrete chain
rule holds up to an error term of order $\Dx (D_{\pm} u_j)^2$.  To a
sequence $\seq{u_j}_{j\in \Z}$ we associate the function $u_\Dx$
defined by
\begin{equation*}
  u_{\Dx}(x)= \sum_{j \in \Z} u_j \mathds{1}_{I_j}(x).
\end{equation*}
Similarly, we also define $r_{\Dx}$ as
\begin{equation*}
  r_{\Dx}(x)= \sum_{j \in \Z} r_j \mathds{1}_{I_j}(x),
\end{equation*}
where $\mathds{1}_{A}$ denotes the characteristic function of the set
$A$.  Throughout this paper we use the notations $u_{\Dx}, r_{\Dx}$ to
denote the functions associated with the sequence $\seq{u_j}_{j\in\Z}$
and $\seq{r_j}_{j\in\Z}$ respectively. For later use, recall that the
$L^{\infty}(\R)$ norm, the $L^1(\R)$ norm, the $L^2(\R)$ norm, and the
$BV(\R)$ semi-norm of a lattice function $u_{\Dx}$ are defined
respectively as
\begin{equation*}
  \begin{aligned}
    & \norm{u_{\Dx}}_{L^{\infty}(\R)} = \sup_{j \in \Z} \abs{u_j}, \\
    & \norm{u_\Dx}_{L^1(\R)} = \Dx \sum_{j\in\Z} \abs{u_j}, \\
    & \norm{u_\Dx}_{L^2(\R)} = \sqrt{\Dx \sum_{j\in\Z} \abs{u_j}^2}, \\
    & \abs{u_{\Dx}}_{BV(\R)} = \sum_{j \in \Z} \abs{u_j - u_{j-1}}.
  \end{aligned}
\end{equation*}
Observe that all the eigenvalues of the system \eqref{eq:system} are
positive by our assumptions. Therefore we consider the following
semi-discrete upwind finite difference scheme
\begin{equation}
  \label{eq:discrete}
  u_j'(t) + \Dmx \left(\phi(r_j(t)) u_j(t)\right)=0, \, \, \text{for} \, \, j \in \Z, \, \, t >0, 
\end{equation}
with initial values
\begin{equation}
  \label{eq:discrete_init}
  u_j(0)=\frac{1}{\Dx}\int_{x_{j-1/2}}^{x_{j+1/2}}\! u_0(x)\,dx,
\end{equation}
where $r_j (t)= \abs{u_j(t)}$.  We have that $\seq{u_j(t)}_{j\in\Z}$
satisfy the (infinite) system of ordinary differential equations and
it is natural to view \eqref{eq:discrete} as an ordinary differential
equation in $L^2(\R)^n$, since the piecewise constant structure of
$u_{\Dx}$ is preserved by the evolution equation
\eqref{eq:discrete}. To show the local (in time) existence and
uniqueness of differentiable solutions we must show that the right
hand side of \eqref{eq:discrete} is Lipschitz continuous in
$L^2(\R)^n$. Set
\begin{equation*}
  F(u_\Dx)_j= \Dm\left(\phi(r_j)u_j\right).
\end{equation*}
The infinite system of differential equations \eqref{eq:discrete} can
then be written
\begin{equation*}
  \frac{d}{dt}\left(u_\Dx(t)\right) = F(u_\Dx)_\Dx.
\end{equation*}
We view $F(u_\Dx)_\Dx$ as an element in $L^2(\R)^n$.  To establish
that this system has a unique solution (at least locally in time) we show
that
\begin{equation}
  \label{eq:lipschitz}
  \norm{F(u_\Dx)_\Dx-F(v_\Dx)_\Dx}_{L^2(\R)^n} \le \gamma
  \norm{u_\Dx-v_\Dx}_{L^2(\R)^n}  
\end{equation}
for some locally bounded $\gamma=\gamma(u_\Dx,v_\Dx)$ and for a fixed
$\Dx>0$.  Set $\tilde{r}_j = \abs{v_j}$, note that
\begin{equation*}
  \abs{r_j-\tilde{r}_j} \le \frac{\abs{u_j+v_j}}{r_j+\tilde{r}_j}
  \abs{u_j-v_j}
  \le \abs{u_j-v_j}.
\end{equation*}
Then
\begin{align*}
  \norm{F(u_\Dx)-F(v_\Dx)}_{L^2(\R)^n} &\le \frac{2}{\Dx}
  \Bigl(\sup_j\abs{u_j} \norm{\phi'}_{L^\infty}
  \norm{r_\Dx-\tilde{r}_\Dx}_{L^2(\R)}\\
  &\qquad \qquad + \norm{\phi}_{L^\infty}
  \norm{u_\Dx-v_\Dx}_{L^2(R)^n} \Bigr)\\
  &\le \gamma \norm{u_\Dx-v_\Dx}_{L^2(R)^n},
\end{align*}
where we have used Assumption~\ref{def:w2}.
Therefore $F$ is locally Lipschitz continuous, and there is a $\tau>0$
so that the initial value problem \eqref{eq:discrete} has a unique
differentiable solution for $t\in [0,\tau)$, if $\tau<\infty$, then
\begin{equation*}
  \lim_{t\uparrow \tau} \norm{u_\Dx(t)}_{L^2(\R)^n} = \infty.
\end{equation*}
We shall proceed to show that the $L^2$ norm remains bounded if it is
bounded initially, so the solution can be defined up to any time.
\begin{lemma}
  \label{lem:rl11}
  Assume that \ref{def:w1}, \ref{def:w2} and \ref{def:w4} hold, and
  let $\seq{u_j(t)}$ be defined by \eqref{eq:discrete}, and let
  $r_j(t) = \abs{u_j(t)}$. Then
  \begin{align*}
    \norm{r_\Dx(t)}_{L^1(\R)} &\le \norm{r_\Dx(0)}_{L^1(\R)}, \\
    \norm{r_\Dx(t)}_{L^2(\R)} &\le \norm{r_\Dx(0)}_{L^2(\R)}, \\
    \norm{r_\Dx(t)}_{L^{\infty}(\R)} &\le
    \norm{r_\Dx(0)}_{L^{\infty}(\R)}.
  \end{align*}
  Furthermore, there is a constant $C$, independent of $\Dx$ and $T$,
  such that
  \begin{equation}
    \label{eq:l2extra}
    \begin{aligned}
      \int_0^T \left(\sum_j \int_{r_{j-1}}^{r_j}\!\!\!
        \left(r_j^2-s^2\right)\phi'(s)\,ds + \Dx\sum_j\phi_{j-1} \Dx
        \abs{\Dm u_j}^2 \right) \,dt\le C.
    \end{aligned}
  \end{equation}
\end{lemma}

\begin{proof}
  Let $\eta$ be a differentiable function $\eta:\R^n\to \R$, take the
  inner product of \eqref{eq:discrete} with $\nabla_u\eta(u_j)$ to get
  \begin{multline}
    \label{eq:entropy1}
    \frac{d}{dt} \eta(u_j) + \Dm \left(\phi_j \eta(u_j) \right) \\
    +
    \left[\left(\nabla_u\eta(u_j),u_j\right)-\eta(u_j)\right]\Dm\phi_j
    + \phi_{j-1} \frac{\Dx}{2} d^2_u \eta_{j-1/2}\left(\Dm u_j,\Dm
      u_j\right)=0.
  \end{multline}
  Here $\phi_j=\phi(r_j)$, and $d^2\eta$ denotes the Hessian matrix of
  $\eta$, so that
  \begin{equation*}
    d^2_u\eta_{j-1/2} = d^2_u\eta(u_{j-1/2})
  \end{equation*}
  for some $u_{j-1/2}$ between $u_j$ and $u_{j-1}$. By a limiting
  argument, the function $\eta(u)=\abs{u}$ can be used. If one
  approximates by convex smooth functions, this means that
  \begin{equation}
    \label{eq:rbnd}
    \frac{d}{dt} r_j + \Dm(r_j\phi(r_j) \le 0.
  \end{equation}
  Multiplying by $\Dx$ and summing over $j$ we get
  \begin{equation}
    \label{eq:l1rbnd}
    \norm{r_\Dx(t)}_{L^1(\R)} \le \norm{u_0}_{L^1(\R)}.
  \end{equation}
  Furthermore, if $j$ is such that $r_{j}(t)\ge r_{j-1}(t)$, then
  since $\phi$ is non-decreasing we get $r_j(t) \phi(r_j(t)) \ge
  r_{j-1}(t) \phi(r_{j-1}(t))$ i.e., $\Dm(r_j \phi(r_j)) \ge
  0$. Hence, from \eqref{eq:rbnd}, we see that $dr_j(t)/dt\le 0$. This
  shows that $0\le r_j(t)\le \sup_j \abs{u_j(0)}$. Hence, if
  $\norm{\abs{u_0}}_{L^\infty(\R)}<\infty$, then $r_\Dx$ is bounded
  independently of $t$ and $\Dx$.

  Choosing $\eta(u)=\abs{u}^2$ in \eqref{eq:entropy1} we get
  \begin{equation*}
    \frac{d}{dt} r_j^2(t) + \Dm \left(r_j^2\phi_j\right) + r_j^2 \Dm
    \phi_j + \phi_{j-1}\Dx\abs{\Dm u_j}^2 = 0.
  \end{equation*}
  We have that
  \begin{align*}
    \Dm\left(r_j^2\phi_j\right) + r_j^2\Dm\phi_j &=
    \frac{2}{\Dx}\int_{r_{j-1}}^{r_j}\!\!\! (s\phi(s) + s^2\phi'(s))
    \,ds + \frac{1}{\Dx}\int_{r_{j-1}}^{r_j}\!\!\!
    \left(r_j^2-s^2\right)
    \phi'(s)\,ds\\
    &=\Dm g(r_j) + \frac{1}{\Dx}\int_{r_{j-1}}^{r_j}\!\!\!
    \left(r_j^2-s^2\right) \phi'(s)\,ds,
  \end{align*}
  where
  \begin{equation}\label{eq:gdef}
    g(r)=2\int_0^r (s\phi(s)+s^2\phi'(s)) \,ds.
  \end{equation}
  Using this we find that
  \begin{equation}
    \label{eq:l2bnd}
    \norm{r_\Dx(t)}_{L^2(\R)} \le \norm{\abs{u_0}}_{L^2(\R)},
  \end{equation}
  since, by the assumption that $\phi'\ge 0$,
  \begin{equation*}
    \int_{r_{j-1}}^{r_j}\!\!\! \left(r_j^2-s^2\right)\phi'(s)\,ds \ge 0.
  \end{equation*}
  Hence $\norm{u_\Dx(t)}_{L^2(\R)^n}$ is bounded independently of
  $\Dx$ and $t$. Therefore, there exists a differentiable solution
  $u_\Dx(t)$ to \eqref{eq:discrete} for all $t>0$. Furthermore, we
  have the bound
  \begin{equation*}
    \int_0^T \left( \sum_j \int_{r_{j-1}}^{r_j}\!\!\! \left(r_j^2-s^2\right)\phi'(s)\,ds  
      + \Dx\sum_j\phi_{j-1} \Dx \abs{\Dm u_j}^2 \right) \,dt\le C,
  \end{equation*}
  for some constant $C$ which is independent of $t$ and $\Dx$.

\end{proof}

Now let $\delta$ be a positive constant, and let $e$ be some unit
vector in $\R^n$. Choose
\begin{equation*}
  \eta(u)=\max\seq{\delta\abs{u} - (e,u),0}.
\end{equation*}
and observe that $(\nabla \eta(u),u)-\eta(u)=0$. Furthermore $\eta$ is
convex, so that
\begin{equation*}
  \frac{d}{dt}\eta(u_j) + \Dm(\eta(u_j)\phi_j) \le 0,
\end{equation*}
which implies that
\begin{equation*}
  \sum_j \eta(u_j(t)) \le \sum_j \eta(u_j(0)).
\end{equation*}
We have that $\eta(u)=0$ if and only if $u$ is in the cone
$\Gamma_\delta=\seq{u \st \delta\abs{u}\le (e,u)}$ for some unit
vector $e$. If $\theta$ denotes the angle between $e$ and $u$, then
$u\in \Gamma_\delta$ if $\cos(\theta)\ge \delta$, thus if $\delta<1$
this is a cone in $\R^n$ and this cone is positively invariant for
\eqref{eq:discrete}. Observe that there is no loss of generality in
choosing the coordinates such that $e=(1,\ldots,1)/\sqrt{n}$. If
$1>\delta>\sqrt{(n-1)/n}$, then the invariant cone is in the first
$2^n$-tant in $\R^n$, so that $\u{i}_j(t)\ge 0$ for all $t>0$ if
$u_0\in \Gamma_\delta$. Since $r_j = \abs{u_j}$, it follows that
$\u{i}_j(t)=0$ if and only if $r_j(t)=0$.

Therefore, if $u_0\in \seq{ u \st \, \abs{u}\le R}\cap \Gamma_\delta$,
then $u_\Dx(x,t)$ is also in this set. This enables us to deduce the
weak-$*$ convergence of a subsequence (which we do not relabel) of
$\seq{u_\Dx}_{\Dx>0}$.

Let now $\eta_i(r)$ and $q_i(r)$ be given by \eqref{eq:entropies} for
$i=1$, $2$. We then have that
\begin{equation}
  \label{eq:eta1}
  \frac{d}{dt}\eta_1(u_j) + \Dm (q_1(r_j)) + e_{1,j} = 0,
\end{equation}
where
\begin{align*}
  f(r)&=r\phi(r),\ \ \   q_1(r)=f(r)-f(k)\ \ \text{and } \\
  e_{1,j}&= \phi_{j-1} \Dx \left(\Dm u_j\right)^T \frac{1}{r_{j-1/2}}
  \left(I-\frac{u_{j-1/2}\otimes
      u_{j-1/2}}{r_{j-1/2}^2}\right)\left(\Dm u_j\right).
\end{align*}
For any vector $u$, the matrix $u \otimes u$ is defined as $(u \otimes
u)_{ij} = u_iu_j$.  We shall now find an equation satisfied by
$\eta_2$. First observe that
\begin{equation*}
  \frac{d}{dt} r_j + f'(r_j)\Dm r_j - \frac{\Dx}{2}
  f''\left(r_{j-1/2}\right) \left(\Dm r_j\right)^2 + e_{1,j}=0.
\end{equation*}
Multiplying this with $f'(u_j)$ we get
\begin{equation}\label{eq:helpeta}
  \frac{d}{dt} f(r_j) + q_2'(r_j) \Dm r_j 
  - f'(r_j)\frac{\Dx}{2}
  f''\left(r_{j-1/2}\right) \left(\Dm r_j\right)^2 + f'(r_j)e_{1,j}=0.
\end{equation}
Set
\begin{equation*}
  e_{2,j}=\frac{\Dx}{2}
  f''\left(r_{j-1/2}\right) \left(\Dm r_j\right)^2.
\end{equation*}
Then \eqref{eq:helpeta} can be rewritten as
\begin{equation}
  \label{eq:eta2}
  \begin{aligned}
    \frac{d}{dt}\eta_2(r_j) + \Dm q_2(r_j) + \frac{\Dx}{2}
    q_2''(r_{j-1/2}) \left(\Dm r_j\right)^2 - f'(r_j) \left(e_{2,j} -
      e_{1,j}\right) = 0.
  \end{aligned}
\end{equation}
Finally set
\begin{equation*}
  e_{3,j} =  \frac{\Dx}{2}
  q_2''(r_{j-1/2}) \left(\Dm r_j\right)^2,
\end{equation*}
and
\begin{equation*}
  e_{i}(x,t)=e_{i,j}(t)\ \text{ for $x\in (x_{j-1/2},x_{j+1/2}]$ and $i=1,2,3$.}
\end{equation*}
\begin{lemma}\label{lem:ei}
  Assume that \ref{def:w1}, \ref{def:w2} and \ref{def:w4} hold, then
  we have that $e_i\in \CMloc(\R\times [0,T))$ for $i=1$, $2$, $3$.
\end{lemma}
\begin{proof}
  Set $\Omega=\R \times [0,T)$, and let $\psi$ be a test function in
  $C_0(\Omega)$. Note that from \eqref{eq:l1rbnd} and \eqref{eq:eta1}
  it follows that
  \begin{equation*}
    \iint_\Omega e_1(x,t)\,dxdt \le C,
  \end{equation*}
  where the constant $C$ does not depend on $\Dx$ or $T$. Since
  $e_1\ge 0$, this means that
  \begin{equation*}
    \abs{\langle e_1,\psi\rangle} \le \iint_\Omega \abs{\psi}e_1
    \,dxdt \le C\norm{\psi}_{L^\infty(\Omega)}, 
  \end{equation*}
  and thus $e_1\in \CMloc(\Omega)$. To show the same for $e_2$ and
  $e_3$ observe that
  \begin{equation*}
    \abs{\Dm r_j}\le \abs{\Dm u_j}.
  \end{equation*}
  Since $\phi(r)>0$, \eqref{eq:l2extra} implies that
  \begin{equation*}
    \int_0^T \Dx\sum_j \Dx\abs{\Dm u_j}^2 \,dt\le C,
  \end{equation*}
  for some constant $C$ which is independent of $\Dx$ and $T$. We also
  have that $f'$ and $f''$ are locally bounded, and $r_\Dx$ is
  bounded, this means that, for $i=2,3$,
  \begin{equation*}
    \iint_\Omega e_i(x,t)\,dxdt \le C \int_0^T \Dx\sum_j \Dx \left(\Dm
      r_j\right)^2 \,dt \le  \int_0^T \Dx\sum_j \Dx\abs{\Dm u_j}^2 \,dt \le C.
  \end{equation*}
  Thus also $e_2$ and $e_3$ are in $\CMloc(\Omega)$.
\end{proof}
Observe that, Lemma~\ref{lem:ei} implies that also
$f'(r_j)(e_{1,j}-e_{2,j})$ is in $\CMloc(\Omega)$.

\begin{lemma}
  \label{lem:compact}
  Assume that \ref{def:w1}, \ref{def:w2} and \ref{def:w4} hold, 
  let $u_{\Dx}$ be generated by the scheme \eqref{eq:discrete} and
  set $r_{\Dx} = \abs{u_{\Dx}}$. Then
  \begin{align*}
    \seq{\eta_i(r_\Dx)_t + q_i(r_\Dx)}_{\Dx>0} \,\, \text{is compact
      in} \,\, H^{-1}_{\loc},
  \end{align*}
  where $\eta_i$ and $q_i$ are given by \eqref{eq:entropies} for $i
  =1,2$.
\end{lemma}

\begin{proof}

  Let $i=1$ or $i=2$, and $\psi$ is a test function in
  $H^1_{\mathrm{loc}}(\Omega)$. we define
  \begin{align*}
    \langle \mathcal{L}_{i}, \psi \rangle & = \langle \eta_i(r_\Dx)_t + q_i(r_{\Dx})_x , \psi\rangle \\
    &= \int_0^T \left(\sum_j \jint \left(\frac{d}{dt}\eta_i(r_j)
        \psi(x,t) -
        q_i(r_j) \psi_x(x,t) \right) \,dx \right) \,dt\\
    &= \int_0^T \left(\sum_j \left( \jint \frac{d}{dt}(\eta_i(r_j)
        \psi(x,t) \,dx -
        q_i(r_j)\, \Dx\Dm \psi(x_{j+1/2},t) \right) \right)\,dt\\
    &= \int_0^T \left( \sum_j \jint \left(\frac{d}{dt}\eta_i(r_j)
        \psi(x,t) +
        \Dm q_i(r_j) \psi(x_{j-1/2},t) \right)\,dx \right) \,dt\\
    &=\int_0^T \left(\sum_j \jint \left(\frac{d}{dt}\eta_i(r_j) + \Dm
        q_i(r_j)
      \right)\psi(x,t)\,dx \right) \,dt \\
    &\qquad + \int_0^T \left( \sum_j \jint
      \left(\psi(x_{j-1/2},t)-\psi(x,t)\right)\Dm q_i(r_j) \,dx \right) \,dt\\
    &=: \langle \mathcal{L}_{i,1},\psi\rangle + \langle
    \mathcal{L}_{2,i},\psi\rangle.
  \end{align*}
  By \eqref{eq:eta1}, \eqref{eq:eta2} and Lemma~\ref{lem:ei} we know
  that $\mathcal{L}_{i,1}\in \CMloc(\Omega)$. Regarding
  $\mathcal{L}_{i,2}$ we have
  \begin{align*}
    & \abs{\langle \mathcal{L}_{2,i},\psi\rangle} = \Bigl| \int_0^T
    \left(\sum_j \jint \int_{x_{j-1/2}}^x \psi_x(y,t)\,dy\, \Dm
      q_i(r_j)\,dx \right) \,dt \Bigr|\\
    &\le \int_0^T \left( \sum_j \jint \sqrt{x-x_{j-1/2}} \Bigl(
      \int_{x_{j-1/2}}^x \left(\psi_x(y,t)\right)^2\,dy\Bigr)^{1/2}
      \abs{\Dm q_i(r_j)} \,dx \right) \,dt \\
    &\le \int_0^T \left(\sum_j \Dx^{3/2} \Bigl( \jint
      \left(\psi_x(x,t)\right)^2\,dx\Bigr)^{1/2}
      \norm{q_i'}_{L^\infty} \abs{\Dm r_j} \right) \,dt \\
    &\le \norm{q_i'}_{L^\infty} \int_0^T \left(\Bigl( \sum_j \Dx \jint
      \left(\psi_x(x,t)\right)^2\,dx \Bigr)^{1/2}
      \Bigl(\Dx^2 \sum_j \left(\Dm r_j\right)^2\Bigr)^{1/2} \right)\,dt\\
    &\le \norm{q_i'}_{L^\infty} \sqrt{\Dx} \Bigl(\iint_\Omega
    \left(\psi_x(x,t)\right)^2\,dxdt \Bigr)^{1/2}
    \Bigl(\int_0^T \Dx\sum_j\Dx \left(\Dm r_j\right)^2 \,dt \Bigr)^{1/2}\\
    &\le C \sqrt{\Dx} \norm{\psi}_{H^1(\Omega)}.
  \end{align*}
  Therefore the above estimate shows that $\mathcal{L}_{2,i}$ is
  compact in $H^{-1}(\Omega)$. By Lemma~\ref{lem:Murat}, we conclude
  the sequence $\seq{\eta_i(r_\Dx)_t + q_i(r_\Dx)}_{\Dx>0}$ is compact
  in $H^{-1}_{\mathrm{loc}}(\Omega)$.

\end{proof}

\begin{lemma}
  \label{lem:rconverg1} Assume that \ref{def:w1} -- \ref{def:w4}
  hold, then there is a subsequence of $\seq{\Dx}$ (not relabeled)
  and a function $r$ such that $r_{\Dx} \to r$ a.e.~$(x,t)\in
  \Omega$. We have that $r\in L^{\infty}([0,T];L^1(\R))$.
  Furthermore, $r$ satisfies
  \begin{equation*}
    \begin{cases}
      r_t + f(r)_x \le 0, &x\in \R,\  t>0,\\
      r=\abs{u_0}, &x\in \R, \ t=0,
    \end{cases}
  \end{equation*}
  in the distributional sense.
\end{lemma}
\begin{proof}
  The strong convergence of $r_\Dx$ follows from the compensated
  compactness theorem, Theorem~\ref{thm:compcomp} and the compactness
  of $\seq{\eta_i(r_\Dx)_t + q_i(r_\Dx)_x}_{\Dx>0}$ for $i=1,2,$ i.e.,
  Lemma ~\ref{lem:compact}.


  To see that $r$ is a distributional subsolution of the conservation
  law \eqref{eq:req}, multiply \eqref{eq:rbnd} with a non-negative
  test function $\psi$ and integrate over $x$ and $t$ to obtain
  \begin{multline*}
    \int_0^T\int_\R r_\Dx \psi_t + f(r_\Dx)\psi_x \,dxdt +
    \int_\R r_\Dx(0,x)\psi(0,x)\,dx \\
    \ge \int_0^T \sum_j f(r_j) \frac{1}{\Dx} \jint \int_x^{x+\Dx}
    \left(\psi_x(x,t)-\psi_x(z,t)\right)\,dz \,dx\,dt.
  \end{multline*}
  The term on the right tends to $0$ as $\Dx\to 0$, which shows that
  $r$ is a subsolution.
\end{proof}
Let now the vector $\tau_j$ be defined as 
\begin{equation*}
  \tau_j = \frac{u_j}{(u_j,e)}, \ \ 
  e=\frac1{\sqrt{n}} (1,\ldots,1), 
\end{equation*}
if $u_j\ne 0$. If $u_j=0$ set $\tau_j =
\tau_{j+1}$. Observe that this makes sense since $\u{i}_j=0$
only if $r_j=0$. Indeed, we have
\begin{equation*}
  \frac{d}{dt} \u{i}_j(t) + \u{i}_{j}(t)\Dm \phi_j + \phi_{j-1}(t)\Dm \u{i}_j
  =0.
\end{equation*}
If $\u{i}_j(t)>0$ for $t<t_0$ and $\u{i}_j(t_0)=0$ then
$d\u{i}_j/dt(t_0)\le 0$. If $\u{i}_{j-1}(t_0)>0$ then $d\u{i}_j/dt
(t_0)>0$, which is a contradiction. Thus if $r_{j_0}(t_0)=0$, then
$r_j(t)=0$ for all $j<j_0$ and $t>t_0$.  Thus the definition of
$\tau^{(i)}_j$ makes sense.

First note that
\begin{equation*}
  \Dm \tau_j = \frac{\left(\Dm u_j\right)\left(u_j,e\right)
    -
    u_j\left(\Dm u_j,e\right)}{\left(u_j,e\right)\left(u_{j-1},e\right)}.
\end{equation*}
Using this, we find
\begin{align*}
  \frac{d}{dt}\tau_j &= \frac{\left(du_j/dt -
      (du_j/dt,e)\right) \left(u_j,e\right) -
    \left(u_j-\left(u_j,e\right)\right)\left(du_j/dt,e\right)}
  {\left(u_j,e\right)^2}\\
  &=-\frac{\Dm\left(u_j\phi_j\right)\left(u_j,e\right) -
    u_j\left(\Dm\left(u_j\phi\right),e\right)}
  {\left(u_j,e\right)^2}\\
  &=-\phi_{j-1} \frac{\left(\Dm u_j\right)\left(u_j,e\right) -
    u_j\left(\Dm u_j,e\right)} {\left(u_j,e\right)^2}\\
  &=-\phi_{j-1}\frac{\left(u_{j-1},e\right)}{\left(u_j,e\right)}
  \Dm\tau_j.
\end{align*}
Set
\begin{equation*}
  \lambda_j =
  \begin{cases}
    \phi_{j-1}\frac{\left(u_{j-1},e\right)}{\left(u_j,e\right)}
    &\text{if $r_j>0$,}\\
    \lambda_{j+1} &\text{if $r_j=0$.}
  \end{cases}
\end{equation*}
Now $\tau_j$ satisfies
\begin{equation}
  \label{eq:taueq}
  \frac{d}{dt} \tau_j + \lambda_j \Dm \tau_j = 0.
\end{equation}
Next, we define for any constant $J>0$ the set
\begin{align*}
\mathcal{M}(J;t) := \lbrace  j \, |\, r_j(t) < r_{j-1}(t) \,\text{and} \, j\Dx \le J \rbrace ,
\end{align*}
and
\begin{align*}
\Gamma_{J}(t) := \min_{j \in \mathcal{M}(J;t)} r_j(t).
\end{align*}
\begin{lemma}
\label{lem:new} 
Assume that the assumption \ref{def:w4} holds. Then we have that $\Gamma_{J}(t) \ge \delta \Gamma_{J}(0) \ge  \delta C_J >0$.
 \end{lemma}
 \begin{proof}
First let $j \in  \mathcal{M}(J;t)$. In particular this means that $r_j(t) < r_{j-1}(t)$. Using scheme \eqref{eq:discrete}, we have
 \begin{align*}
 \frac{d}{dt} (u_j, e) & = -\left( \Dm(u_j \phi_j), e \right) \\
 &= - \Dm \phi_j (u_j, e) - \phi_{j-1} (\Dm u_j, e).
 \end{align*}
 Therefore, if $(u_j,e)$ is a local minima, i.e., if $(u_j, e) <
 (u_{j-1}, e)$ then $ \frac{d}{dt} (u_j, e) \ge 0$.  This concludes
 the proof 
 since $\delta C_J \le \delta \Gamma_{J}(0) \le \delta r_j(0) \le
 (u_j(0), e) \le (u_j(t), e) \le r_j(t)$. Observe that this proof also implies that the set $\mathcal{M}(J;t)$ can't increase in time.
 \end{proof}

\begin{lemma}
  If the assumptions \ref{def:w1} -- \ref{def:w5} hold, and if 
  \begin{equation}\label{eq:tauinbv}
    \abs{\tau^{(i)}_\Dx(\cdot,0)}_{B.V.(\R)} \le C, \ \ i=1,\ldots,n,
  \end{equation}
  for some constant $C$ which is independent of $\Dx$, then there is a
  subsequence of $\seq{\Dx}$ (not relabeled) and functions
  $\tau^{(i)}$ in $C([0,T];L^1_\loc(\R)$ such that
  $\tau^{(i)}_\Dx(\cdot,t)\to \tau^{(i)}(\cdot,t)$ in $L^1_\loc(\R)$
  for $i=1,\ldots,n$.
\end{lemma}
\begin{remark}
  If $\tau(\cdot,0)$ is of bounded variation, or if $u_0$ is of
  bounded variation, and satisfy \ref{def:w4} and \ref{def:w5}, then 
  \eqref{eq:tauinbv} holds.
\end{remark}
\begin{proof}
  Note that $\lambda_j\ge 0$, and that $\lambda_j$ is bounded in any compact interval.  We only want to show that $\lambda_j$ is bounded in any compact interval of $\R$, since Kolmogorov's compactness theorem gives the convergence in $C([0,T];L^1_{\mathrm{loc}}(\R))$. First observe that
  \begin{align*}
   \abs{ \phi_{j-1}\frac{\left(u_{j-1},e\right)}{\left(u_j,e\right)}} \le \frac{\phi_{j-1}}{\delta} \frac{r_{j-1}}{r_j}.
  \end{align*}
  Now if $r_j \ge r_{j-1}$, then clearly $\lambda_j$ is bounded. Again if $r_j < r_{j-1}$, thanks to Lemma~\ref{lem:new}, we have $\lambda_j$ is bounded.
  
  Set
  $\theta_j=\Dm \tau^{(i)}_j$.  Then $\theta_j$ satisfies
  \begin{equation}\label{eq:thetadef}
    \frac{d}{dt}\theta_j + \lambda_{j-1}\Dm\theta_j + \theta_j \Dm \lambda_{j}=0.
  \end{equation}
  Let $\eta_\alpha(\theta)$ be a smooth approximation to
  $\abs{\theta}$ such that
  \begin{equation*}
    \eta''_\alpha(\theta)\ge 0 \ \text{ and }\ \lim_{\alpha\to
      0}\eta_\alpha(\theta) = \lim_{\alpha\to
      0}\left(\theta\eta'_\alpha(\theta)\right)=\abs{\theta}.
  \end{equation*}
  We multiply \eqref{eq:thetadef} by $\eta'_\alpha(\theta_j)$ to get
  an equation satisfied by $\eta_\alpha(\theta_j)$. Observe that
  \begin{align*}
    \lambda_{j-1}\eta_\alpha'(\theta_j)\Dm\theta_j +
    \theta_j\eta_\alpha'(\theta_j)\Dm \lambda_j &=
    \lambda_{j-1}\Dm\eta_\alpha(\theta_j)
    +\theta_j\eta'_\alpha(\theta_j)
    \Dm \lambda_j\\
    &\qquad \qquad + \frac{\Dx}{2}\lambda_{j-1}
    \eta''_\alpha(\theta_{j-1/2})\left(\Dm\theta_j\right)^2 \\
    &\ge
    \Dm\left(\lambda_j\eta_\alpha(\theta_j)\right) \\
    &\qquad \qquad +
    \left(\theta_j\eta'_\alpha(\theta_j)-\eta_\alpha(\theta_j)\right)\Dm
    \lambda_j.
  \end{align*}
  Hence
  \begin{equation*}
    \frac{d}{dt}\eta_\alpha(\theta_j) +
    \Dm\left(\lambda_j\eta_\alpha(\theta_j)\right) \le
    \left(\eta_\alpha(\theta_j)-\theta_j\eta'_\alpha(\theta_j)\right)\Dm \lambda_j.
  \end{equation*}
  Now let $\alpha\to 0$ to obtain
  \begin{equation}
    \label{eq:abstheta}
    \frac{d}{dt}\abs{\theta_j} + \Dm\left(\lambda_j\abs{\theta_j}\right)
    \le 0.
  \end{equation}
  If we multiply this with $\Dx$, sum over $j$ and integrate in $t$,
  we find that
  \begin{equation}
    \label{eq:bv}
    \begin{aligned}
      \abs{\tau^{(i)}_\Dx(\cdot,t)}_{B.V} \le
      \abs{\tau^{(i)}_{\Dx}(\cdot,0)}_{B.V.} \le C
    \end{aligned}
  \end{equation}
  Note that, since $\tau^{(i)}_{\Dx}(\cdot,t)$ has bounded variation
  and satisfies \eqref{eq:taueq}, it is $L^1_{\mathrm{loc}}$ Lipschitz
  continuous in $t$, that is
  \begin{equation}
    \label{eq:timecont}
    \begin{aligned}
      \norm{\tau^{(i)}_{\Dx}(\cdot,t)-\tau^{(i)}_\Dx(\cdot,s)}_{L^1([-J, J])}
      \le \sup_{j\Dx \in [-J,J]} \lambda_j \abs{\tau^{(i)}(\cdot,0)}_{B.V.} \abs{t-s},
    \end{aligned}
  \end{equation}
  for any compact interval $[-J,J]$ of $\R$.
  Hence, the above estimates \eqref{eq:bv}, \eqref{eq:timecont} and an
  application of Kolmogorov's compactness criterion (Lemma
  ~\ref{lem:kolmogorov}) shows that $\tau^{(i)}=\lim_{\Dx\to
    0}\tau^{(i)}_{\Dx}$ is continuous in $t$, with values in
  $L^1_{\mathrm{loc}}(\R)$. In other words, the convergence is in
  $C([0,T];L^1_{\mathrm{loc}}(\R))$.
\end{proof}
Now we have the strong convergence of $r_\Dx$ and of
$\tau_\Dx$. This means that also $u_\Dx$ converges strongly to
some function $u$ in $C([0,T];L^1_{\mathrm{loc}}(\R))$ since we have
\begin{equation}
  \label{eq:trigonometry}
  u=\tau (u,e) = \frac{\tau}{\abs{\tau}} r.
\end{equation}
\begin{theorem}
  \label{thm:convergence} Assume that \ref{def:w1} -- \ref{def:w5} and
  \eqref{eq:tauinbv} hold.  Let $u_\Dx$ be defined by
  \eqref{eq:discrete} -- \eqref{eq:discrete_init}.  Then there exists
  a function $u$ in $L^{\infty}([0,T];L^1_{\mathrm{loc}}(\R))$ and a
  subsequence $\seq{\Dx_j}$ of $\seq{\Dx }$ such
  that $u_{\Dx_j}\to u$ as $\Dx_j\to 0$. The function $u$ is a weak
  solution to \eqref{eq:system}.
\end{theorem}
\begin{proof}
  We have already established convergence.

  It remains to show that $u$ is a weak solution. To this end, observe
  that\footnote{Here we ``extend'' the definition of $D_-$ and $D_+$
    to arbitrary functions in the obvious manner.}
  \begin{equation*}
    \int_0^T \int_\R \Dm\left(u_{\Dx_j}\phi(r_{\Dx_j})\right) \psi(x,t)\,dxdt
    = -  \int_0^T \int_\R u_{\Dx_j}\phi(r_{\Dx_j}) D_+\psi(x,t)\,dxdt.
  \end{equation*}
  As $\Dx_j\to 0$, $D_+\psi\to \psi_x$ for any $\psi \in
  C^1_0(\Omega)$. This means that $u$ is a weak solution.
\end{proof}

\section{Fully Discrete Schemes}
\label{sec:fully}
In this section, we propose three different fully discrete schemes and
show two of them converge to the unique entropy solution of
\eqref{eq:system}. We start by introducing some notations needed to
define the fully discrete finite difference schemes. We reserve $\Dt$
to denote a small positive number that represent the temporal
discretization parameter of the numerical schemes. For $n=0,1,\cdots,
N$, where $N \Dt =T$, for some fixed time horizon $T>0$, we set $t^n =
n \Dt$. For any function $v(t)$, admitting pointvalues, we let
$D^t_{+}$ denote the discrete forward difference operator in the time
direction, i.e.,
\begin{equation*}
  \Dpt v(t) = \frac{v(t+\Dt)-v(t)}{\Dt}. 
\end{equation*}
Furthermore, we introduce the spatial-temporal grid cells
\begin{align*}
  I_j^n = [ x_{j-1/2}, x_{j+1/2}) \times [t^n, t^{n+1}).
\end{align*}
As before, to a sequence $\seq{u^n_j}_{j\in \Z, n\ge 0}$ we associate
the function $u_\Dx$ defined by
\begin{equation*}
  u_{\Dx}(x,t)= \sum_{j \in \Z, n\ge 0} u^n_j \mathds{1}_{I^n_j}(x,t),
\end{equation*}
similarly, we also define $r_{\Dx}$ as
\begin{equation*}
  r_{\Dx}(x,t)= \sum_{j \in \Z, n\ge 0} r^n_j \mathds{1}_{I^n_j}(x,t).
\end{equation*}

First, we consider the following fully discrete finite difference
scheme
\begin{equation}
  \label{eq:fullydiscrete}
  \Dpt u^n_j + \Dm\left(u^n_j\phi\left(\abs{u^n_j}\right)\right) = 0, 
\end{equation}
with initial values
\begin{equation*}
  u^0_j = \frac{1}{\Dx}\jint u_0(x)\,dx.
\end{equation*}

We start by proving the following lemma.
\begin{lemma}
  \label{lem:discrete_weak_bv} Suppose $u_0 \in L^2(\R)^n \cap
  L^{\infty}(\R)^n$, and
  \begin{equation*}
    \lambda \le
    \min\seq{\frac{1}{\norm{f'}_{L^\infty}},
      \frac{1}{C_\phi\left(1+\norm{u_0}_{L^\infty(\R)}\right)^2}}, 
  \end{equation*}
  hold. Then
  \begin{equation}
    \label{eq:1l2bnd}
    \begin{aligned}
      \norm{u_{\Dx}(\cdot,t_n)}_{L^2(\R)^n}^2 & \le \norm{u_0}_{L^2(\R)^n}^2, \\
      \norm{u_{\Dx}(\cdot,t_n)}_{L^{\infty}(\R)^n}^2 & \le
      \norm{u_0}_{L^{\infty}(\R)^n}^2,
    \end{aligned}
  \end{equation}
  for all $n>0$, furthermore
  \begin{equation}
    \label{eq:l2weak_bv}
    \Dt\Dx \sum_{n=0,j\in \Z}^{N-1} \Dx \abs{\Dm u^n_j}^2 \le 2\norm{u_0}^2_{L^2(\R)^n},
  \end{equation}
  where $\lambda = \frac{\Dt}{\Dx}$ and $N \Dt=T$.
\end{lemma}

\begin{proof}
  For any (differentiable) function $\eta:\R^n \to \R$, we can take
  the inner product of \eqref{eq:fullydiscrete} with
  $\nabla_u\eta(u^n_j)$ and obtain
  \begin{multline}
    \label{eq:entropyfully}
    \Dpt \eta\left(u^n_j\right) +
    \Dm\left(\eta\left(u^n_j\right)\phi^n_j\right)\\
    + \left(\left(\nabla_u\eta\left(u^n_j\right),u^n_j\right)
      -\eta\left(u^n_j\right)\right)\Dm \phi^n_j \\
    + \frac{\Dx}{2}\phi^n_{j-1} {\Dm u^n_j}^T d^2_u\eta^n_{j-1/2}\Dm
    u^n_j -\frac{\Dt}{2}{\Dpt u^n_j}^T d^2\eta^{n+1/2}_j \Dpt u^n_j =
    0
  \end{multline}
  As before, we choose $\eta(u)=\abs{u}^2$ to get
  \begin{multline*}
    \Dpt \abs{u^n_j}^2 + \Dm\left(\abs{u^n_j}^2\phi^n_j\right)
    + \abs{u^n_j}^2\Dm \phi^n_j \\
    + \phi_{j-1}^n\Dx\abs{\Dm u^n_j}^2 -\Dt\abs{\Dm u^n_j\phi^n_j}^2 =
    0.
  \end{multline*}
  The upper line in the above formula can be rewritten as
  \begin{equation*}
    \Dpt \abs{u^n_j}^2 +  \Dm\left(g^n_j\right) +
    \frac{1}{\Dx}\int_{r^n_{j-1}}^{r^n_j}\!\!
    \left(\left(r^n_j\right)^2-s^2\right)\phi'(s)\, ds,
  \end{equation*}
  where $g$ is defined in \eqref{eq:gdef}. In order to balance the two
  last terms we proceed as follows.
  \begin{align*}
    \phi_{j-1}&\Dx\abs{\Dm u^n_j}^2 -\Dt\abs{\Dm u^n_j\phi^n_j}^2
    \\
    &\ge \phi(0) \Dx \left(\abs{\Dm u^n_j}^2 -
      \lambda\abs{\phi^n_{j-1} \Dm u^n_j + u^n_j \Dm \phi^n_j}^2
    \right)
    \\
    &\ge \phi(0)\Dx\biggl( \abs{\Dm u^n_j}^2
    \\
    &\hphantom{\ge \Dx\phi(0)\quad} - \lambda
    \left(\left(\left(1+\abs{u^n_{j-1}}\right)\max_{r}\phi'(r)\right)^2
      \abs{\Dm u^n_j}^2 + \abs{u^n_j}^2 \max_r\phi'(r)\abs{\Dm
        u^n_j}^2\right)
    \biggr)\\
    &\ge \phi(0)\Dx \abs{\Dm u^n_j}^2
    \left(1-C_\phi\lambda\left(1+\abs{u^n_j}\right)^2\right),
  \end{align*}
  where the constant $C_\phi$ only depends on $\phi'$. We have the
  obvious inequality
  \begin{equation*}
    \abs{u^n_j}^2 \le \frac{1}{\Dx}\norm{u_{\Dx}(\cdot,t_n)}_{L^2(\R)^n}^2.
  \end{equation*}
  We shall not use this, instead we assume that $\lambda$ is so small
  that
  \begin{equation}
    \label{eq:cfl1}
    1-C_\phi\lambda\left(1+\sup_j r^n_j\right)^2\ge \frac{1}{2}.
  \end{equation}
  If this holds, setting $R=r^2$, then we have that
  \begin{align*}
    R^{n+1}_j &\le R^n_j - \lambda\left(g\left(\sqrt{R^n_j}\right)-
      g\left(\sqrt{R^n_{j-1}}\right)\right)\\
    &=R^{n}_j - \lambda
    \frac{g'\left(\sqrt{R^n_{j-1/2}}\right)}{2\sqrt{R^n_{j-1/2}}}
    \left(R^n_j-R^n_{j-1}\right)
    \\
    &=R^n_j - \lambda f'(r^n_{j-1/2}) \left(R^n_j-R^n_{j-1}\right)\\
    &=\left(1-\lambda f'(r^n_{j-1/2}\right) R^n_j + \lambda
    f'(r^n_{j-1/2})R^n_{j-1},
  \end{align*}
  where $f(r)=r\phi(r)$. Hence if
  \begin{equation}
    \label{eq:cfl2}
    \lambda \norm{f'}_{L^\infty}<1,
  \end{equation}
  then $R^{n+1}_j$ is dominated by a convex combination of $R^n$ and
  $R^n_{j-1}$, and thus $\sup_j\seq{\abs{u^{n+1}_j}}\le \sup_j
  \seq{\abs{u^n_j}}$. Therefore, if we assume that
  \begin{equation}
    \label{eq:cfl_final}
    \lambda \le
    \min\seq{\frac{1}{\norm{f'}_{L^\infty}},
      \frac{1}{C_\phi\left(1+\norm{u_0}_{L^\infty(\R)}\right)^2}}, 
  \end{equation}
  we have that $r^n_j \le r^0_j\le \norm{u_0}_{L^\infty(\R)}$ for all
  $n$ and $j$. Hereafter, \eqref{eq:cfl_final} is always assumed to
  hold.

  Under the CFL-condition, \eqref{eq:cfl_final}, the equation for
  $R^n_j$ can be written
  \begin{equation*}
    \Dpt R^n_j + \Dm G(R^n_j) + \frac{\Dx}{2} \abs{\Dm
      u^n_j}^2 \le 0,
  \end{equation*}
  where $G(R)=g(\sqrt{R})$.  Summing this over $n=0,\ldots,N-1$ yields
  \begin{equation*}
    \Dx \sum_j R^N_j + \Dx\Dt\frac12 \sum_{n,j}\Dx\abs{\Dm u^n_j}^2
    \le \Dx \sum_j R^0_j \le \int_\R \abs{u_0}^2\,dx. 
  \end{equation*}
  This finishes the proof of the lemma.
\end{proof}
Rewriting yet again the scheme for $R$, we have
\begin{equation}\label{eq:Reqn}
  \Dpt R^n_j + \Dm G\left(R^n_j\right) = e_{1,j}^n + e_{2,j}^2,
\end{equation}
where
\begin{align*}
  e_{1,j}^n &=-\frac{1}{\Dx} \int_{r_{j-1}^n}^{r^n_j} \!\!
  \left(\left(r^n_j\right)^2 - s^2\right)\phi'(s)\,ds\\
  e_{2,j}^n &= -\phi_{j-1}^n\Dx\abs{\Dm u^n_j}^2 +\Dt\abs{\Dm
    \left(u^n_j\phi^n_j\right)}^2.
\end{align*}
Let us define
\begin{equation*}
  e_{i,\Dx}=e^n_{i,j} \ \ \text{for $x\in (x_{j-1/2},x_{j+1/2}]$ and 
    $t \in [t_n,t_{n+1})$, and $i=1,2$.}
\end{equation*}
\begin{lemma}
  \label{lem:cmloc_discrete}
Assume that \ref{def:w1}, \ref{def:w2} and \ref{def:w4} hold, then
  we have that $e_{i,\Dx}\in \CMloc(\Omega)$ for $i=1,2$.
\end{lemma}
\begin{proof}
  First we observe that there is a constant $C$ (independent of $n$,
  $j$ and $\Dx$ such that
  \begin{equation*}
    \abs{e_{2,j}^n}\le C\Dx\abs{\Dm u^n_j}^2.
  \end{equation*}
  We also have that
  \begin{equation*}
    \abs{e^n_{1,j}}\le \frac{C}{\Dx} \left(r^n_j-r_{j-1}^n\right)^2
    \left(r^n_j+r^n_{j-1}\right) \le C\Dx \abs{\Dm u^n_j}^2,
  \end{equation*}
  for some constant $C$ depending of $\phi$ and
  $\norm{u_0}_{L^\infty(\R)}$. Then, for a test function $\psi\in
  L^\infty(\Omega)$,
  \begin{equation*}
    \iint_\Omega e_{i,\Dx}\psi\,dxdt \le
    C\norm{\psi}_{L^\infty(\Omega)}
    \iint_\Omega \Dx \abs{\Dm u_{\Dx}}^2 \,dxdt \le 
    C\norm{\psi}_{L^\infty(\Omega)},
  \end{equation*}
  for $i=1,2$. This finishes the proof.
\end{proof}
Multiplying \eqref{eq:Reqn} by $G'(R^n_j)$ we find
\begin{multline}
  \label{eq:eta2disc}
  \Dpt G^n_j + \Dm q_{2,j}^n =
  G'\left(R^n_j\right)\left(e^n_{1,j}+e^2_{2,j}\right) \\
  +\frac{\Dt}{2} G''\left(R^{n+1/2}_j\right)\left(\Dpt R^n_j\right)^2\\
  +\frac{\Dx}{2} G''\left(R^n_{j-1/2}\right)\left(\Dm R^n_j\right)^2
  -\frac{\Dx}{2} q_{2}''\left(\tilde{R}^n_{j-1/2}\right) \left(\Dm
    R^n_j\right)^2,
\end{multline}
where
\begin{equation*}
  q_2(R)=\int_k^R \left(G'(\theta)\right)^2\,d\theta.
\end{equation*}
We find that
\begin{equation*}
  G''(R)=\frac12 \phi''(r)+ \frac{\phi'(r)}{r}.
\end{equation*}
This is potentially unbounded as $r\downarrow 0$, hence we shall
assume that $\phi$ is such that
\begin{equation}
  \label{eq:phiassumption}
  \lim_{r\downarrow 0} \frac{\phi'(r)}{r} < \infty.
\end{equation}
We are going to show that also the right hand side of
\eqref{eq:eta2disc} is in $\CMloc(\Omega)$. First observe that since
$r^n_j$ is bounded,
\begin{align*}
  \left(\Dm R^n_j\right)^2
  &=\left[\left(r^n_j+r^n_{j-1}\right)\Dm r^n_j\right]^2\\
  &\le C\left(\Dm r^n_j\right)^2\\
  &\le C\abs{\Dm u^n_j}^2,
\end{align*}
and similarly $(\Dpt R^n_j)^2 \le \abs{\Dpt u^n_j}^2$. Defining
\begin{align*}
  e^n_{3,j}&= \frac{\Dt}{2} G''\left(R^{n+1/2}_j\right)\left(\Dpt R^n_j\right)^2\\
  e^n_{4,j}&= +\frac{\Dx}{2}
  \left(G''\left(R^n_{j-1/2}\right)\left(\Dm R^n_j\right)^2 -
    q_{2}''\left(\tilde{R}^n_{j-1/2}\right) \left(\Dm
      R^n_j\right)^2\right)
\end{align*}
By the assumption \eqref{eq:phiassumption} $G''$ and $q_2''$ are
locally bounded, hence we have that
\begin{equation*}
  \abs{e_{i,j}^n} \le C\Dx \abs{\Dm u^n_j}^2, \ \ \text{for $i=3,4$,}
\end{equation*}
and for some $C$ which is independent of $\Dx$. By the same argument
used to prove Lemma~\ref{lem:cmloc_discrete}, also $e_{3,\Dx}$ and
$e_{4,\Dx}$ are in $\CMloc(\Omega)$. We now have established that
\begin{equation*}
  \Dpt \eta^n_{i,j}+ \Dm q^n_{i,j} = E^n_{i,j}, \ \ \text{$i=1,2$,}
\end{equation*}
where
\begin{equation*}
  \eta^n_{1,j}=\left(R^n_j-k\right), \ \ \eta^n_{2,j}=G(R^n_j)-G(k),
\end{equation*}
and $q^n_{i,j}$ are given by the corresponding fluxes
\begin{equation*}
  q^n_{1,j}=G(R^n_{i,j})-G(k), \ \ q^n_{2,j}=\int_{k}^{R^n_j} G'(\theta)^2\,d\theta.
\end{equation*}
The terms $E^n_{i,j}$ are such that the corresponding piecewise
continuous functions $E_{i,\Dx}$ are in $\CMloc(\Omega)$.
\begin{lemma}
  \label{lem:rconverg} Suppose \eqref{eq:phiassumption} and the
  assumptions \ref{def:w1} -- \ref{def:w4} hold.
  Then there is a subsequence of $\seq{R_\Dx}$ (not relabeled) and a
  function $R$ such that $R_{\Dx} \to R$ a.e.~$(x,t)\in \Omega$. We
  have that $R\in L^\infty([0,T];L^1(\R))$.  Furthermore, $R$
  satisfies
  \begin{equation*}
    \begin{cases}
      R_t + G(R)_x \le 0, &x\in \R,\  t>0,\\
      R=\abs{u_0}^2, &x\in \R, \ t=0,
    \end{cases}
  \end{equation*}
  in the distributional sense.
\end{lemma}
\begin{proof}
  The condition \ref{def:w3} means that if we can show that
  $(\eta_{i,\Dx})_t + (q_{i,\Dx})_x$ for $i=1,2$ is compact in
  $H^{-1}_\loc(\Omega)$, $R_\Dx$ converges strongly. To show this, let
  $\psi$ be a smooth function in $H^1_\loc(\Omega)$. We have that
  \begin{align*}
    \langle (\eta_{i,\Dx})_t & + (q_{i,\Dx})_x, \psi \rangle =
    \sum_{n=0,j\in \Z} \! \iint_{I^n_j} \eta^n_{i,j} \psi_t +
    q^n_{i,j}\psi_x \,dxdt\\
    &=\sum_{n,j} \jint \eta^n_{i,j} \left(\psi\left(t_{n+1},x\right)
      -\psi\left(t_n,x\right)\right)\,dx \\
    &\hphantom{\sum_{i,j}}\quad + \tint q^n_{i,j}
    \left(\psi\left(x_{j+1/2},t\right)-\psi\left(x_{j-1/2},t\right)\right)\,dt
    \\
    &=-\sum_{n,j}\left(\eta^{n+1}_{i,j}-\eta^n_{i,j}\right)\jint
    \psi\left(t_{n+1},x\right)\,dx \\
    &\hphantom{-\sum_{n,j}}\quad +
    \left(q^{n}_{i,j}-q^n_{i,j-1}\right)
    \tint \psi\left(x_{j-1/2},t\right)\,dt\\
    &\qquad - \sum_j\jint \eta^0_{i,j}\psi(x,0)\,dx\\
    &=-\sum_{n,j} \left(\Dpt \eta^n_{i,j} + \Dm q^n_{i,j}\right)
    \iint_{I^n_j} \psi(x,t)\,dxdt - \int_\R
    \eta_{i,\Dx}(x,0)\psi(x,0)\,dx\\
    &\qquad + \sum_{n,j} \Dpt \eta^n_{i,j} \iint_{I^n_j}
    \psi(x,t)-\psi\left(t_{n+1},x\right)\,dxdt \\
    &\qquad + \sum_{n,j} \Dm q^n_{i,j} \iint_{I^n_j}
    \psi(x,t)-\psi\left(x_{j-1/2},t\right)\,dxdt\\
    &= -\sum_{n,j} E^n_{i,j} \iint_{I^n_j} \psi(x,t)\,dxdt - \int_\R
    \eta_{i,\Dx}(x,0)\psi(x,0)\,dx + \alpha_1 + \alpha_2.
  \end{align*}
  Hence, $(\eta_{i,\Dx})_t + (q_{i,\Dx})_x$ consists of the sum of
  $\alpha_1$ and $\alpha_2$ and a term which is in
  $\CMloc(\Omega)$. We must show that $\alpha_1$ and $\alpha_2$ are
  compact in $H^{-1}(\Omega)$. First, observe that since $R^n_j$ is
  uniformly bounded,
  \begin{equation*}
    \left(\Dpt \eta^n_{i,j}\right)^2 \le C \abs{\Dm u^n_j}^2\ \text{ and }\ 
    \left(\Dm q^n_{i,j}\right)^2 \le C\abs{\Dm u^n_j}^2. 
  \end{equation*}
  Using this
  \begin{align}
    \alpha_1 &\le C\sum_{n,j} \abs{\Dm u^n_j} \iint_{I^n_j}
    \int_t^{t_{n+1}} \abs{\psi_t(x,\tau)}\,d\tau\, \,dxdt\notag\\
    &\le C\sum_{n,j} \abs{\Dm u^n_j} \iint_{I^n_j} \sqrt{t_{n+1}-t}
    \Bigl( \tint
    \left(\psi_t(x,\tau)\right)^2\,d\tau\Bigr)^{1/2}\,dxdt\notag\\
    &\le C\sum_{n,j}\abs{\Dm u^n_j} \Dt^{3/2}\sqrt{\Dx}
    \Bigl(\iint_{I^n_j} \left(\psi_t(x,t)\right)^2\,dxdt\Bigr)^{1/2}\notag\\
    &\le C\sqrt{\Dt} \biggl( \Dt\Dx \sum_{n,j} \Dx\abs{\Dm
      u^n_j}^2\biggr)^{1/2} \biggl( \sum_{n,j}\iint_{I^n_j}
    \left(\psi_t(x,t)\right)^2\,dxdt\biggr)^{1/2}\notag\\
    &\le C\sqrt{\Dx}\norm{\psi}_{H^1(\Omega)}.\label{eq:smallterm}
  \end{align}
  Hence $\alpha_1$ is compact in $H^{-1}(\Omega)$, that $\alpha_2$ is
  compact follows by analogous arguments. Now we have established the
  strong convergence (along a subsequence which we do not relabel) of
  $R_{\Dx}$.

  By the CFL-condition, \eqref{eq:cfl_final}, the right hand side of
  \eqref{eq:Reqn} is negative, hence
  \begin{equation*}
    \Dpt R^n_j + \Dm G(\R^n_j) \le 0.
  \end{equation*}
  Multiplying this with a non-negative test function $\psi$, doing a
  summation by parts and then sending $\Dx$ to zero, using the same
  arguments that led to \eqref{eq:smallterm}, yields that the limit
  $R$ is a subsolution in the distributional sense.
\end{proof}
Since $R_\Dx$ converges strongly to $R$, also $r_\Dx$ will converge
strongly to $r:=\sqrt{R}$. The sequence $\seq{u_\Dx}_{\Dx>0}$ is
uniformly bounded, so a subsequence will converge weak-$*$ to some
function $u\in L^{\infty}(\Omega)$. By using the arguments leading up
to \eqref{eq:smallterm} it is straightforward to show that
\begin{equation*}
  \iint_\Omega u \psi_t + u\phi(r) \psi_x \,dxdt + \int_\R
  u_0(x)\psi(x,0)\,dx = 0,
\end{equation*}
for all $\psi\in C^\infty_0(\Omega)$. Hence the limit $u$ is a
distributional solution of
\begin{equation*}
  u_t + \left(u\phi(r)\right)_x = 0.
\end{equation*}
In order to conclude that $u$ is a weak solution to \eqref{eq:system},
we would have to show that $\abs{u}=r$. We have not been able to prove
this, and merely conclude that $\abs{u}\le r$. The reason for this is
that $v\mapsto \abs{v}$ is convex, and that weak limits of a convex
function are not less than the convex function of the weak limit.

To overcome this difficulty, we propose another fully discrete scheme
based on explicit decoupling of the variables $r$ and $w$.
\subsection{A scheme which enforces the entropy condition}
\label{subsec:afully}
Define
\begin{equation*}
  w_{\Dx} =
  \begin{cases}
    \frac{u_{\Dx}}{r_{\Dx}}, \qquad &\text{if} \, \, r_{\Dx} \neq 0,  \\
    0, \qquad &\text{if} \, \, r_{\Dx} = 0,
  \end{cases}
\end{equation*}
and let $r_\Dx$ and $w_\Dx$ satisfy
\begin{equation}
  \label{eq:rfully}
  \begin{cases}
    r^{n+1}_j = r^n_j - \Dt \Dm f^n_j, &n\ge 0,\\
    r_j^0 = \abs{u^0_j},
  \end{cases}
\end{equation}
and
\begin{equation}
  \label{eq:vfully}
  \begin{cases}
    w^{n+1}_j = w^n_j - \Dt \phi^n_j \Dm w^n_j, &n\ge 0,\\
    r^0_j w^0_j = u^0_j.
  \end{cases}
\end{equation}
To ensure the convergence of the approximations $\seq{r_\Dx}$ we
choose
\begin{equation}
  \label{eq:cflr}
  \Dt \norm{f'}_{L^\infty(\R)} \le \Dx.
\end{equation} 
We list some useful properties of $r_{\Dx}$ in the next lemma
\cite{Holden}.
\begin{lemma}
  \label{lem:ensol}
  Assume that the CFL condition \eqref{eq:cflr} holds and $r_0 \in
  BV(\R) \cap L^{\infty}(\R)$. Then for each $\Dx >0$ we have that
  \begin{itemize}
  \item [(a)] $-M \le r_{\Dx}(x,t) \le M$, for all $x$ and $t>0$.
  \item [(b)] For $n \ge 0$ the functions
    \begin{align*}
      n \mapsto \Dx \sum_{j \in \Z} \abs{r^n_j}, \quad n \mapsto
      \sum_{j \in \Z} \abs{r^n_j - r^n_{j-1}}, \quad n \mapsto \sum_{j
        \in \Z} \abs{r^{n+1}_j - r^n_j}
    \end{align*}
    are non-increasing. In particular this means that the family
    $\lbrace{r^{\Dx}\rbrace}_{\Dx>0}$ is (uniformly in $\Dx$) bounded
    in $L^{\infty}(\R^{+}; L^1(\R)) \cap \mathrm{BV} (\R \times
    \R^{+})$.
  \item [(c)] Moreover $r_\Dx(\cdot,t) \to r(\cdot,t)$ strongly in
    $L^1(\R)$ for all $t\ge 0$, where $r\in
    \mathrm{Lip}([0,T];L^1(\R))$ and is the unique entropy (in the
    sense of Kru\v{z}kov) solution of the conservation law
    \begin{equation}
      \label{eq:scalcons}
      \begin{cases}
        r_t + f(r)_x = 0, & \\
        r(x,0)= r_0.&
      \end{cases}
    \end{equation}
  \end{itemize}
\end{lemma}

Observe that, we can write the scheme \eqref{eq:vfully} as
\begin{equation*}
  w^{n+1}_j = \left(1-\lambda\phi^n_j\right) w^n_j + \lambda \phi^n_j w^n_{j-1}.
\end{equation*}
If $\lambda \phi^n_j < 1$ for all $j$, then $w^{n+1}_j$ is a convex
combination of $w^n_j$ and $w^n_{j-1}$. Thus
\begin{equation}
  \label{eq:vbnds}
  \inf_j w_j^0 \le w^n_j \le \sup_j w^0_j, \quad n>0.
\end{equation}
and
\begin{equation}
  \label{eq:bv_1}
  \begin{aligned}
    \abs{w_\Dx(\cdot,t)}_{B.V.(\R)^n} \le
    \abs{w_{\Dx}(\cdot,0)}_{B.V.(\R)^n}.
  \end{aligned}
\end{equation}
Furthermore
\begin{equation}
  \label{eq:timecont_1}
  \begin{aligned}
    \Dx \sum_j \abs{w^{n+1}_j - w^n_j} \le \Dt \norm{\phi}_{L^\infty}
    \sum_j \abs{w^n_j - w^n_{j-1}} \le C\Dt
    \abs{w_{\Dx}(\cdot,t_n)}_{B.V.(\R)^n}.
  \end{aligned}
\end{equation}
Hence the map $t\mapsto w_{\Dx}(\cdot,t)$ is $L^1$- Lipschitz
continuous. Finally, the above estimates \eqref{eq:bv_1},
\eqref{eq:timecont_1} and an application of Kolmogorov's compactness
criterion (Lemma~\ref{lem:kolmogorov}) shows that $w=\lim_{\Dx\to
  0}{w_\Dx}$, where the convergence is along a
subsequence. Furthermore $w\in C([0,T];(L^1_\loc(\R))^n)$.

Multiply the equation \eqref{eq:rfully} for $r^{n+1}_j$ with that
\eqref{eq:vfully} for $w^{n+1}_j$ to get
\begin{align*}
  r^{n+1}_j w^{n+1}_j &= \left(r^n_j - \Dt\Dm f^n_j \right)
  \left(w^n_j - \Dt \phi^n_j \Dm w^n_j\right) \\
  &= r^n_j w^n_j - \Dt\left( w^n_j \Dm f^n_j + f^n_j \Dm w^n_j\right)
  + \Dt^2 \phi^n_j\Dm f^n_j \Dm w^n_j\\
  &=r^n_j w^n_j - \Dt \left( w^n_j \Dm f^n_j + f^n_{j-1}\Dm w^n_j
  \right) - \Dt \left(f^n_j - f^n_{j-1}\right) \Dm w^n_j\\
  &\qquad \quad + \Dt^2
  \phi^n_j\Dm f^n_j \Dm w^n_j \\
  &=r^n_j w^n_j - \Dt \Dm \left(f^n_j w^n_j \right) +
  \Dt\left(f^n_j-f^n_{j-1}\right) \Dm w^n_j \left(\lambda\phi^n_j -
    1\right).
\end{align*}
Then we have
\begin{equation}\label{eq:rvdisc}
  \Dpt \left(r^n_jw^n_j\right) + \Dm \left(f^n_j w^n_j\right) = \Dt
  \left(\lambda\phi^n_j-1\right)\left(f^n_j-f^n_{j-1}\right)\Dm w^n_j =: e^n_j.
\end{equation}
Let now $\psi\in C^\infty_0(\Omega)$ be a test function, multiply the
above equation by $\psi$ and integrate over $\Omega$ to get
\begin{align*}
  \sum_{n=1,j}^\infty \int_{t_n}^{t_{n+1}}& \jint r^n_j w^n_j \Dmt
  \psi +
  f^n_j w^n_j \Dp \psi \,dxdt\\
  &+ \frac{1}{\Dt} \int_0^\Dt\sum_j \jint r^0_jw^0_j \psi \,dxdt =
  \sum_{n,j} \int_{t_n}^{t_{n+1}} \jint e^n_j \psi \,dxdt.
\end{align*}
Since we have the convergence of $r_\Dx$ and $w_\Dx$, the left hand
side of this converges to
\begin{equation*}
  \iint_\Omega r w \psi_t + f(r) w \psi_x \,dxdt + \int_\R r(x,0)w(x,0)
  \psi(x,0)\,dx. 
\end{equation*}
Regarding the right hand side we have
\begin{align*}
  \Bigl| \sum_{n,j} \int_{t_n}^{t_{n+1}} & \jint e^n_j
  \psi \,dxdt\Bigr| \\
  &\le \Dt \norm{\psi}_{L^\infty(\Omega)}
  \left(\lambda\norm{\phi}_{L^\infty} + 1\right) \Dt \sum_{n,j}
  \abs{f^n_{j}-f^n_{j-1}}\,\abs{w^n_j -w^n_{j-1}}\\
  &\le \Dt C
  \norm{\psi}_{L^\infty(\R)}\norm{w_\Dx(x,0)}_{L^\infty(\R)^n} T
  \norm{r_{\Dx}(x,0)}_{B.V.(\R)},
\end{align*}
where $T$ is such that $\mathrm{supp}\psi\subset [0,T]$. Hence
\begin{equation*}
  \iint_\Omega r w \psi_t + f(r) w \psi_x \,dxdt + \int_\R r(x,0)w(x,0)
  \psi(x,0)\,dx=0.
\end{equation*}
Hence, we see that $r w$ is a weak solution to the Cauchy problem
\begin{equation*}
  (rw)_t + \left( f(r) w \right)_x = 0.
\end{equation*}
In other words, $(r,rw)$ is a weak solution to
\begin{equation}
  \label{eq:eqrrw}
  \begin{aligned}
    r_t + f(r)_x & =0, \\
    (r\,w)_t + (\phi(r) r\,w)_x &=0.
  \end{aligned}
\end{equation}
Next, we shall make use of Lemma ~\ref{lem:wagnar} to conclude that
$\abs{w} =1$.

Finally, collecting all the results above, we have proved the
following theorem.
\begin{theorem}
  \label{thm:final}
  Assume that $u_0\in B.V.(\R)$. If $\lambda=\Dt/\Dx$ satisfies the
  CFL-condition $\lambda<\sup_x f'(|u_0(x)|)$, and $u_\Dx$ is defined
  by \eqref{eq:rfully}, \eqref{eq:vfully}, then $u=\lim_{\Dx\to
    0}u_\Dx$ is the unique entropy (in the sense of
  Definition~\ref{def:entropy}) solution to \eqref{eq:system}.
\end{theorem}
\begin{remark}
  Observe that to establish the convergence of the entire sequence
  $\seq{u_\Dx}_{\Dx\ge 0}$, we first use Kolmogorov's compactness
  criterion to establish convergence of a subsequence, then, since the
  entropy solution $u$ is unique, every subsequence of $\seq{u_\Dx}$ will
  contain a subsequence converging to $u$. This means that the entire
  sequence converges.
\end{remark}

We propose another scheme based on discretizing the ``conservative''
form \eqref{eq:req}--\eqref{eq:veq1}.  Let $r_\Dx$ and $u_\Dx$ satisfy
\begin{equation}
  \label{eq:rfully_1}
  \begin{cases}
    r^{n+1}_j = r^n_j - \Dt \Dm f(r^n_j), &n\ge 0,\\
    r_j^0 = \abs{u_j^0},
  \end{cases}
\end{equation}
and
\begin{equation}
  \label{eq:ufully}
  u^{n+1}_j=u^n_j - \Dt \Dm\left(u^n_j \phi(r^n_j)\right), 
\end{equation}
for $n\ge 0$ and $ f(r) =r \phi(r)$, with $u^0_j$ given by
\eqref{eq:discrete_init}.  Again as before, regarding the convergence
of the approximations $\seq{r_\Dx}$ we have Lemma ~\ref{lem:ensol}.
On the other hand, we can rewrite the scheme for $u^n_j$ as
\begin{equation*}
  u^{n+1}_j = u^n_j + \lambda \left(u^n_j\phi^n_j -
    u^n_{j-1}\phi^n_{j-1}\right)=:F_j^n\left(u^n_j,u^n_{j-1}\right).
\end{equation*}
Each component
$F_j^{(i),n}(u^n_j,u^n_{j-1})=F_j^n(u^{(i),n}_j,u^{(i),n}_{j-1})$ is
linear and monotone in both its arguments. Furthermore, the scheme for
$r^n_j$ can be written
\begin{equation*}
  r^{n+1}_j = F_j^n\left(r^n_j,r^n_{j-1}\right).
\end{equation*}
Set $w^n_j=r^n_j - u^{(i),n}_j$. Then, by the linearity of $F^n_j$,
\begin{align*}
  w^{n+1}_j &= F\left(r^n_j,r^n_{j-1}\right) -
  F\left(u^{(i),n}_j,u^{(i),n}_{j-1}\right) \\
  &=F\left(r^n-u^{(i),n}_j,r^n_{j-1}-u^{(i),n}_{j-1}\right)\\
  &=F\left(w^n_j,w^n_{j-1}\right).
\end{align*}
Therefore $\inf_j w^0_{j}\le w^{n+1}_j\le \sup_j w^{0}_j$. We have
that $w^0_j = \abs{u^0_j}-u^{(i),0}_j\ge 0$, therefore
\begin{equation}
  \label{eq:ubnd2}
  \abs{u_\Dx(x,t)}\le r_\Dx(x,t)\le \abs{u_0(x)}.
\end{equation}
Hence, the sequence $\seq{u_\Dx}_{\Dx>0}$ is uniformly bounded, and
there converges $L^\infty$ weak-$*$ (modulo a subsequence) to some
function $u$. A straightforward computation yields that
\begin{multline*}
  \biggl|\iint_\Omega u_\Dx \psi_t + u_\Dx \phi\left(r_\Dx\right)
  \phi_t
  \,dxdt + \int_\R u_0(x)\psi(x,0)\,dx \biggr|\\
  \le \sum_{n=1,j} \abs{u^n_j} \Bigl| \iint_{I^n_j} \left(\Dmt \psi -
    \psi_t\right)\,dxdt\Bigr| + \abs{u^n_j}\phi\left(r^n_j\right)
  \Bigl|\iint_{I^n_j} \left(D_+\psi
    -\psi_x\right)\,dxdt\Bigr|\\
  + \frac{1}{\Dt}\Bigl|\int_0^\Dt \int_\R u_\Dx(x,0)\psi(x,t)-\Dt
  u_0(x)\psi(x,0)\,dx\Bigr|
  \\
  \le C\Dx\left(\norm{\psi_{tt}}_{L^\infty(\Omega)}+
    \norm{\psi_{xx}}_{L^\infty(\Omega)}+\norm{\psi_t}_{L^\infty(\Omega)}\right).
\end{multline*}
Sending $\Dx$ to zero, we see that the limit $u$ is a weak solution to
the Cauchy problem
\begin{equation*}
  u_t + \left(u\phi(r)\right)_x = 0.
\end{equation*}
Again we shall make use of Lemma ~\ref{lem:wagnar} to conclude that
$\abs{u} =r$.

Finally, We have proved the following theorem.
\begin{theorem}
  \label{thm:final2}
  Assume that $u_0\in B.V.(\R)$. If $\lambda=\Dt/\Dx$ satisfies the
  CFL-condition $\lambda<\sup_x f'(|u_0(x)|)$, and $u_\Dx$ is defined
  by \eqref{eq:rfully_1}, \eqref{eq:ufully}, then $u=\lim_{\Dx\to
    0}u_\Dx$ is the unique entropy (in the sense of
  Definition~\ref{def:entropy}) solution to \eqref{eq:system}.
\end{theorem}

\section{Numerical experiments}
\label{sec:numerical}
We close this paper by demonstrating how these schemes work in
practice. We perform all the computations for $2 \times 2$ system with
$\phi(r)=r^2$.
\subsection{Numerical Experiment 1}
\label{sec:numer1}
In this case we approximate the system \eqref{eq:system} with initial
data
\begin{equation}\label{eq:numrieinit}
  U_0(x) =
  \begin{cases}
    U_l, & \quad x<0, \\
    U_r, & \quad x>0.
  \end{cases}
\end{equation}
It is not difficult to find the exact solution of \eqref{eq:system} in
this case. For the sake of completeness we write the explicit form of
the exact solutions $U(x,t) = \bar{U}(x/t)$.

If $\abs{U_l} < \abs{U_r}$, then
\begin{equation}\label{eq:numriesol1}
  \bar{U}(\xi) =
  \begin{cases}
    U_l, & \quad \xi \le \abs{U_l}^2, \\
    U_m, & \quad \abs{U_l}^2 \le \xi \le 3 \abs{U_l}^2 , \\
    (\frac{\xi}{3})^{1/2} \frac{U_r}{\abs{U_r}}, & \quad 3 \abs{U_l}^2 \le \xi \le 3 \abs{U_r}^2 \\
    U_r, & \quad \xi \ge 3 \abs{U_r}^2,
  \end{cases}
\end{equation}

If $\abs{U_l} > \abs{U_r}$, then
\begin{equation*}
  \bar{U}(\xi) =
  \begin{cases}
    U_l, & \quad \xi \le \abs{U_l}^2, \\
    U_m, & \quad \abs{U_l}^2 \le \xi \le  \abs{U_l}^2 + \abs{U_l} \abs{U_r} + \abs{U_r}^2 , \\
    U_r, & \quad \xi \ge \abs{U_l}^2 + \abs{U_l} \abs{U_r} +
    \abs{U_r}^2,
  \end{cases}
\end{equation*}
with $U_m = \frac{\abs{U_l}}{\abs{U_r}} U_r$ in both cases.

In what follows, we test the fully discrete explicit numerical scheme
\eqref{eq:fullydiscrete} with initial data
\begin{equation*}
  U_0(x) =
  \begin{cases}
    U_{-}, & \quad x<0, \\
    U_{+}, & \quad x>0,
  \end{cases}
\end{equation*}
where
\begin{align*}
  U_{-} = (0.5,1.5), \qquad U_{+} =(1.5,2.0),
\end{align*}
for the first experiment and
\begin{align*}
  U_{-} = (1.5,2.0), \qquad U_{+} =(0.5,1.5),
\end{align*}
for the second experiment.  The computations are performed on a
computational domain $[-1,20]$ with $4000$ mesh points. To enforce the
CFL condition we set the time step $\Dt = (CFL) \Dx/3
\sup{\abs{U_0}^2}$, where we use a CFL number $0.75$. Although we do
not plot the exact solutions, a comparison of the computational
results displayed in Figs ~\ref{fig:1} with the exact solution shows
good agreement.

\begin{figure}[htbp]
  \centering \subfigure[$T
  =0$]{\includegraphics[height=0.4\linewidth,width=0.49\linewidth]{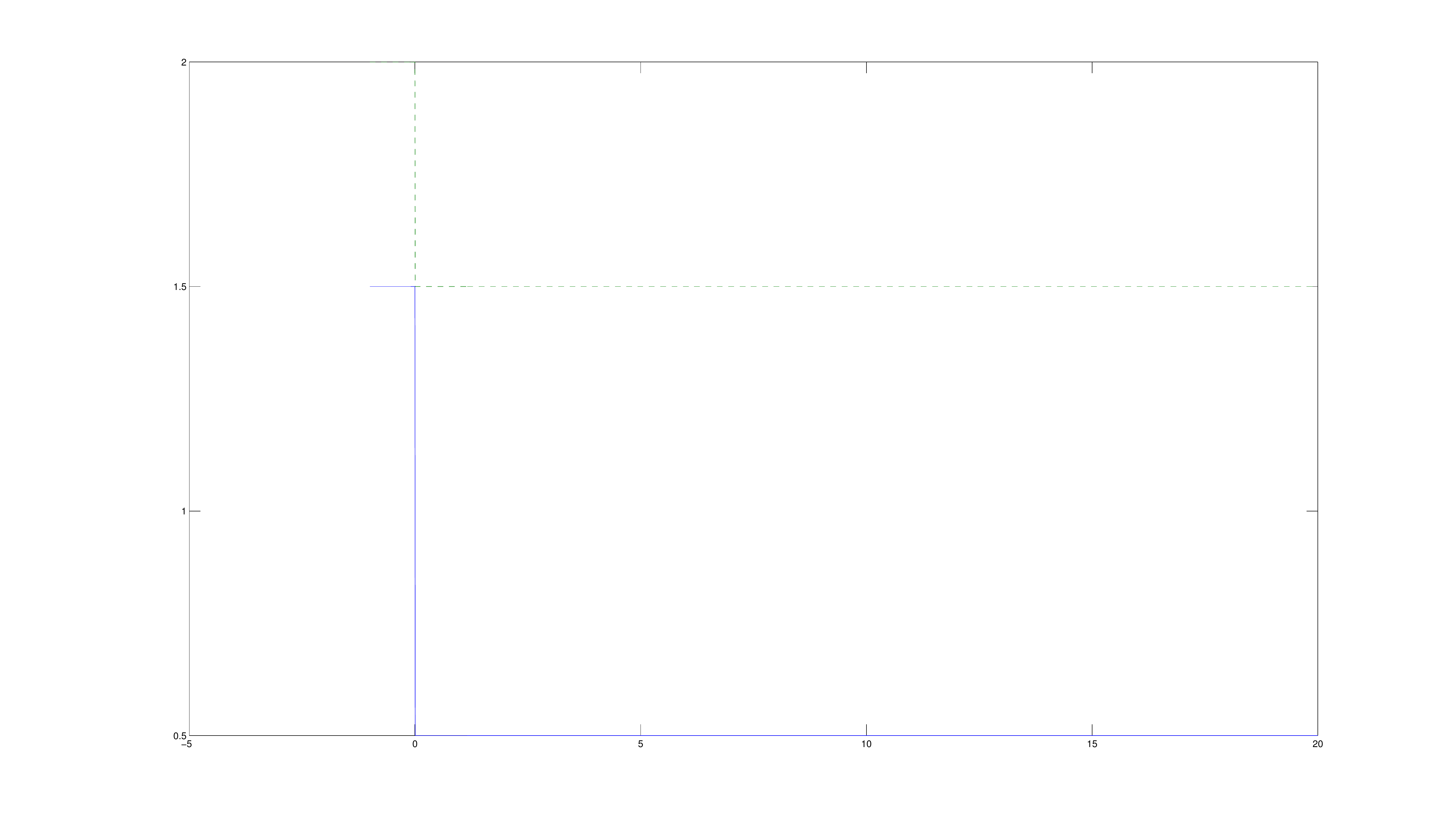}}
  \subfigure[$T = 0$]{\includegraphics[height=0.4\linewidth,width=0.49\linewidth]{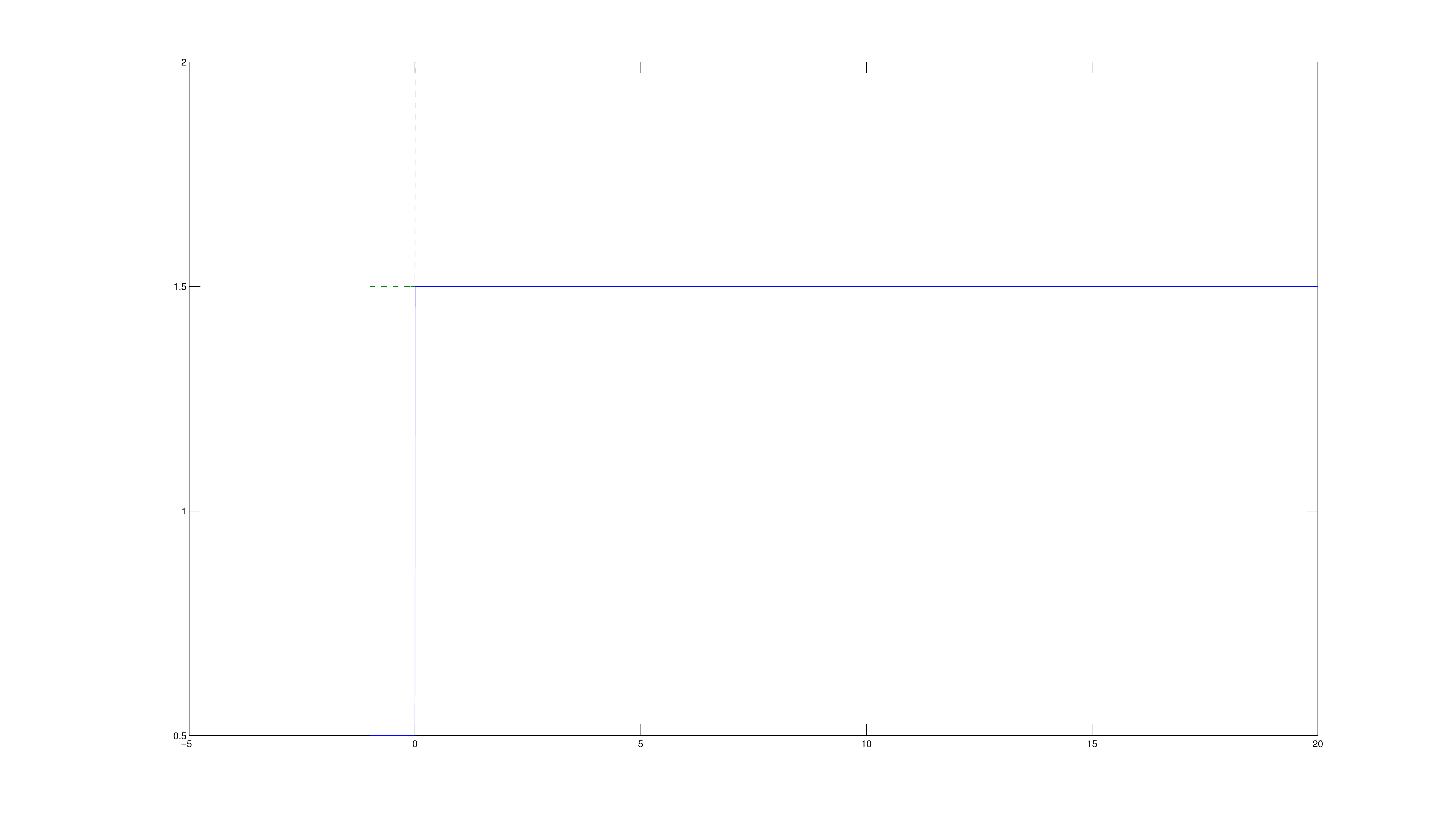}}\\
  \subfigure[$T
  =0.50$]{\includegraphics[height=0.4\linewidth,width=0.49\linewidth]{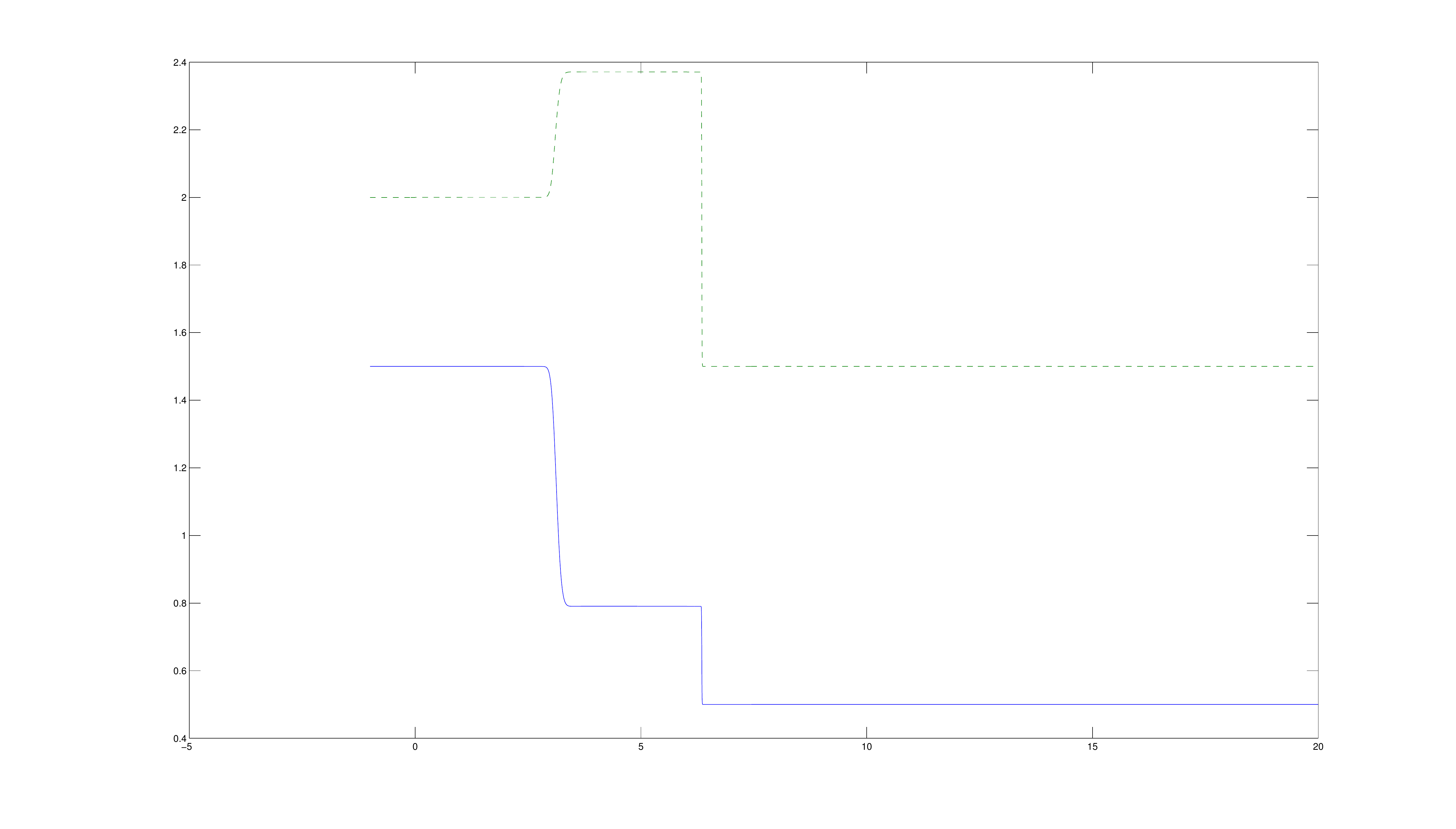}}
  \subfigure[$T
  =0.50$]{\includegraphics[height=0.4\linewidth,width=0.49\linewidth]{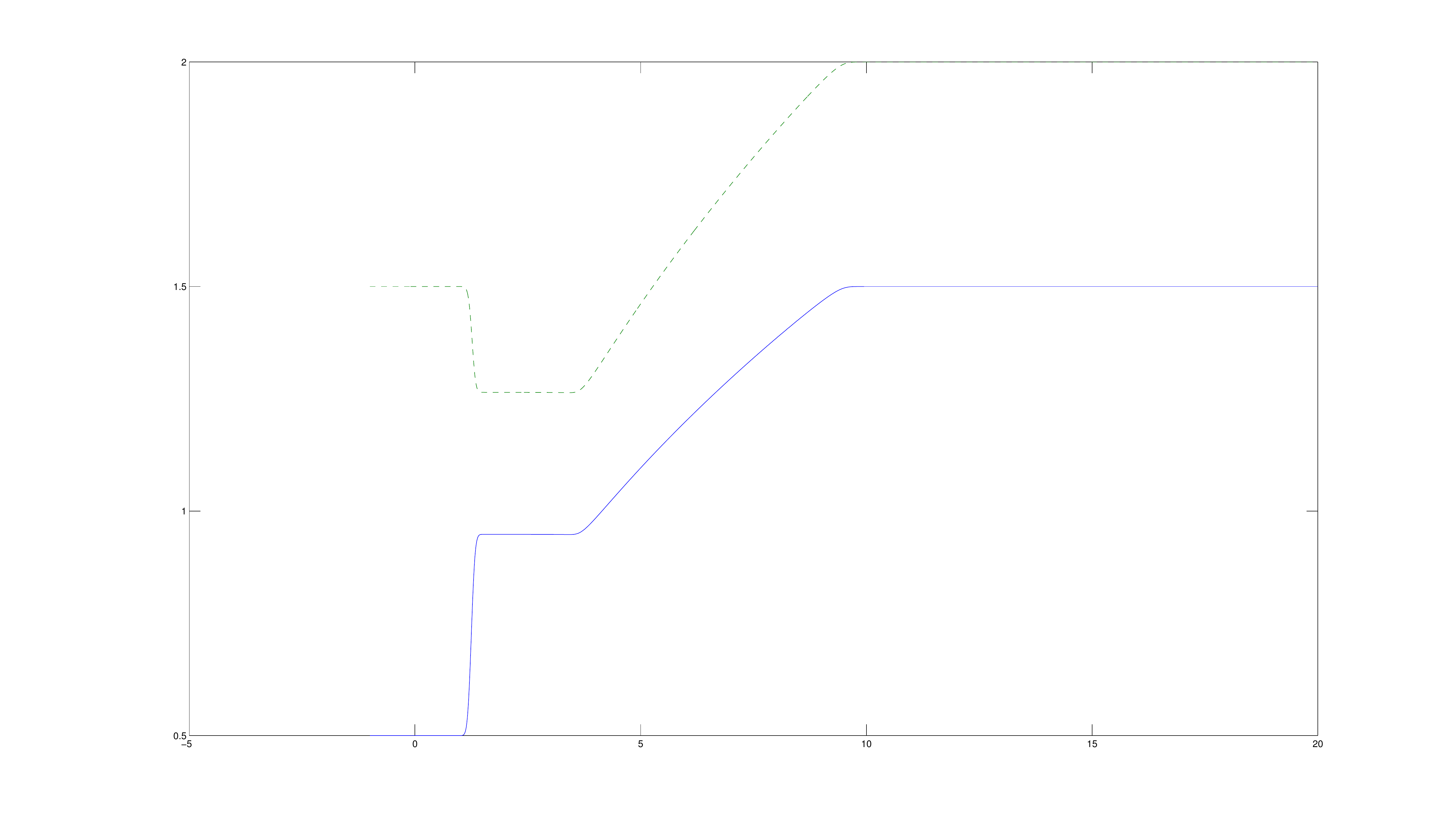}}
  \caption{ Left column: Experiment-$1$: The dotted-dashed curve
    represents the first component of $U$, the solid curve represents
    the second component. Right column: Experiment-$2$: The
    dotted-dashed curve represents the first component of $U$, the
    solid curve represents the second component.}
  \protect \label{fig:1}
\end{figure}

\subsection{Numerical experiment 2}
\label{numer:2}
In this case, we test our fully discrete explicit numerical scheme
\eqref{eq:rfully}--\eqref{eq:vfully} with initial data $U_0 = r_0
w_0$, where
\begin{equation*}
  r_0(x) =
  \begin{cases}
    r_{-}, & \quad x<0, \\
    r_{+}, & \quad x>0,
  \end{cases}
\end{equation*}
with
\begin{align*}
  r_{-} = 1.0, \qquad r_{+} =0.75,
\end{align*}
for the first and third numerical experiments and
\begin{align*}
  r_{-} = 0.75, \qquad r_{+} =1.0,
\end{align*}
for the second and fourth numerical experiments.  Similarly, for $w_0$
we take
\begin{equation*}
  w_0(x) =
  \begin{cases}
    (1.0, 0.0), & \quad x<0.2 \\
    (\cos(8 \pi (x-0.2)), \sin(8 \pi (x-0.2))), & \quad 0.2 \le x \le 0.7, \\
    (1.0, 0.0), & \quad x \ge 0.7,
  \end{cases}
\end{equation*}
for the first and second numerical experiments and
\begin{equation*}
  w_0(x) =
  \begin{cases}
    (1.0, 0.0),  & \quad x \le 0.2, \\
    (-1.0, 0.0), & \quad x \ge 0.2,
  \end{cases}
\end{equation*}
In this case also, it is easy to find the exact solution. Although we
do not plot the exact solutions, we give the explicit form of the
exact solution. The exact solution is given by $U = r w$ with
\begin{equation*}
  r(x,t) =
  \begin{cases}
    r_{-}, & \quad x \le s\,t, \\
    r_{+}, & \quad x \ge s\,t,
  \end{cases}
  \quad \text{with}\,\, s = r_{-}^2 + r_{-} r_{+} + r_{+}^2,
\end{equation*}
and
\begin{equation*}
  w(x,t) =
  \begin{cases}
    w_0 (x - r_{-}^2 t), & \quad x \le r_{-}^2 t, \\
    w_0( \frac{r_{-}}{r_{+}} (x - r_{-}^2 t)), & \quad r_{-}^2 t \le x \le  s\,t , \\
    w_0( x - r_{+}^2 t), & \quad x \ge s\,t,
  \end{cases}
\end{equation*}
for the first and third numerical experiments. Similarly,
\begin{equation*}
  r(x,t) =
  \begin{cases}
    r_{-}, & \quad x \le 3 r_{-}^2 t, \\
    (x/3t)^{1/2}, & \quad 3 r_{-}^2 t \le x \le 3 r_{+}^2 t \\
    r_{+}, & \quad x \ge 3 r_{+}^2 t,
  \end{cases}
\end{equation*}
and
\begin{equation*}
  w(x,t) =
  \begin{cases}
    w_0 (x - r_{-}^2 t), & \quad x \le r_{-}^2 t, \\
    w_0( \frac{r_{-}}{r_{+}} (x - r_{-}^2 t)), & \quad r_{-}^2 t \le x \le 3 r_{-}^2 t , \\
    w_0( \frac{2}{3 \sqrt{3} r_{+}} x^{3/2} t^{-1/2}), & \quad 3 r_{-}^2 t \le x \le 3 r_{+}^2 t \\
    w_0( x - r_{+}^2 t), & \quad x \ge 3 r_{+}^2 t,
  \end{cases}
\end{equation*}
for the second and fourth numerical experiments.

In all the experiments computational domain is $[-1,4]$ and we use
Neumann boundary conditions at the left boundary. We also use a $CFL$
number $0.75$ and $4000$ mesh points for all the experiments.  A
comparison of the computational results displayed in Figs
~\ref{fig:2}--\ref{fig:3} with the exact solution shows good
agreement.

Below we show the computational results for four different
qualitatively significant sets of data: a compression or an expansion
wave in $r$ initiated slightly behind a continuous pulse or a
discontinuous contact wave in $w$. Fig ~\ref{fig:2}--\ref{fig:3}
display the computed solution at three different times. In the plots,
the dot-dash curve represents the first component of $U$ and the
dotted curve represents the second component, while the solid curve
represents the $r$-component of $(r,U)$.
\begin{figure}[htbp]
  \centering \subfigure[$T
  =0$]{\includegraphics[height=0.4\linewidth,width=0.49\linewidth]{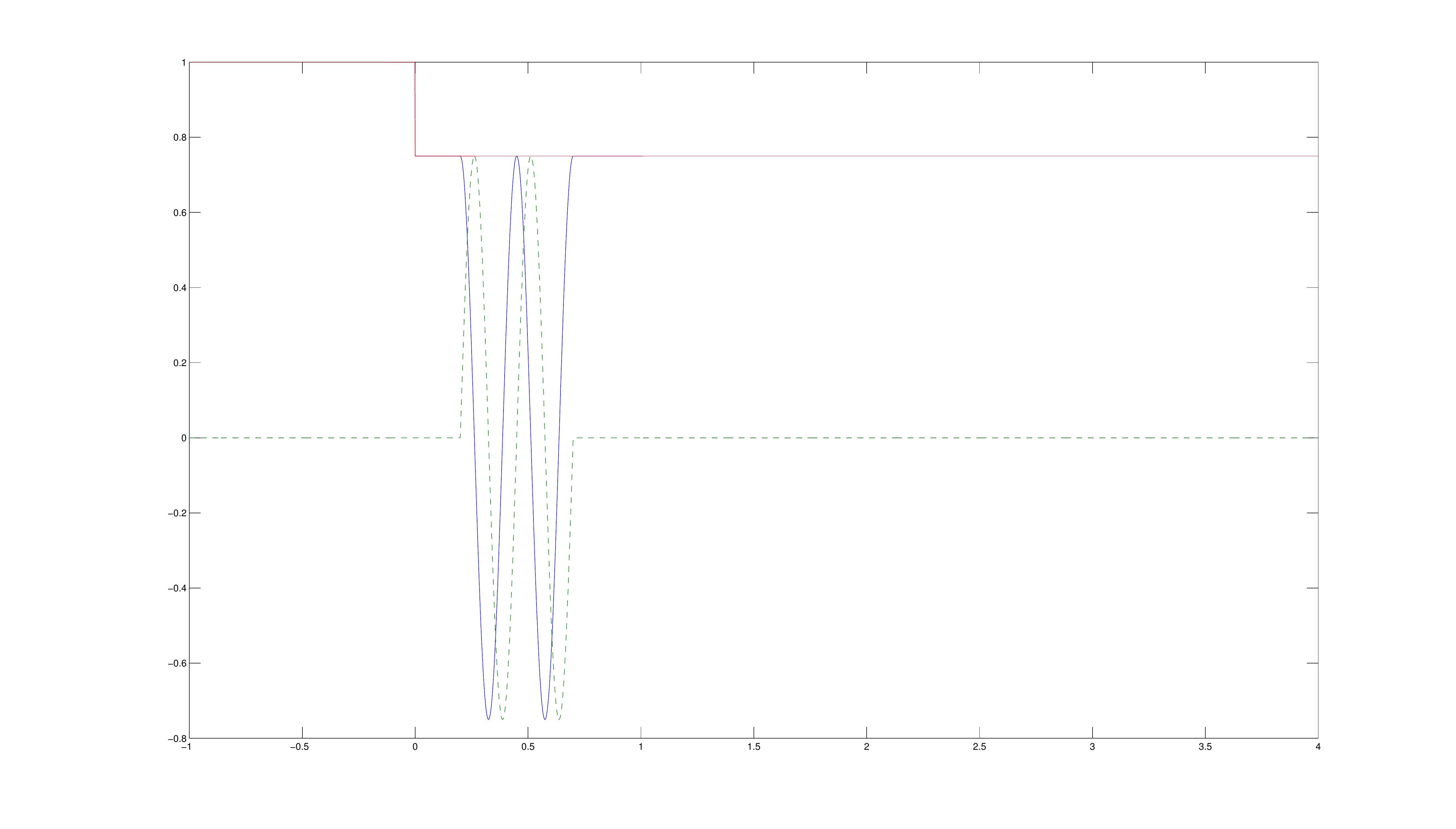}}
  \subfigure[$T = 0$]{\includegraphics[height=0.4\linewidth,width=0.49\linewidth]{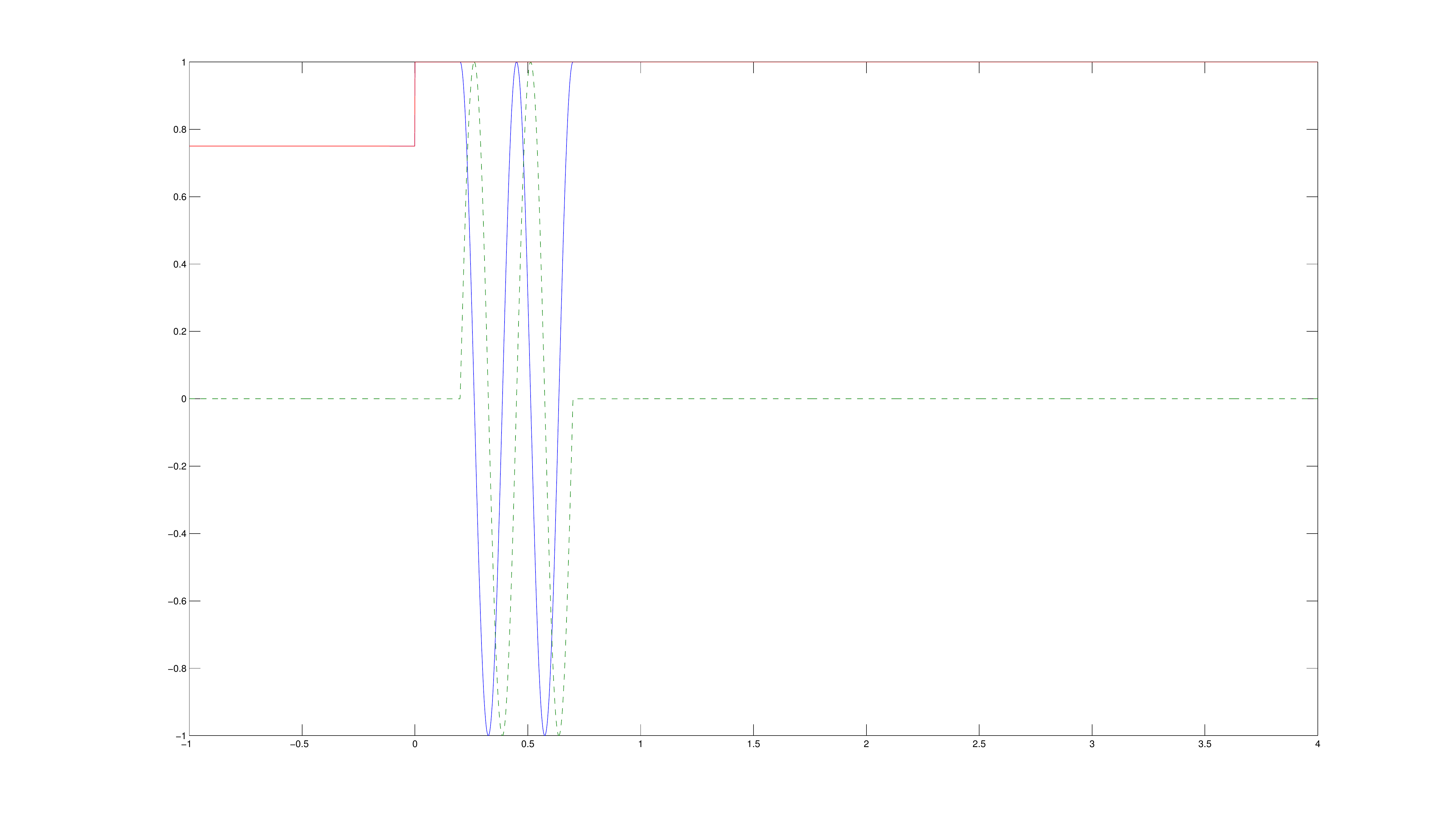}}\\
  \subfigure[$T
  =0.25$]{\includegraphics[height=0.4\linewidth,width=0.49\linewidth]{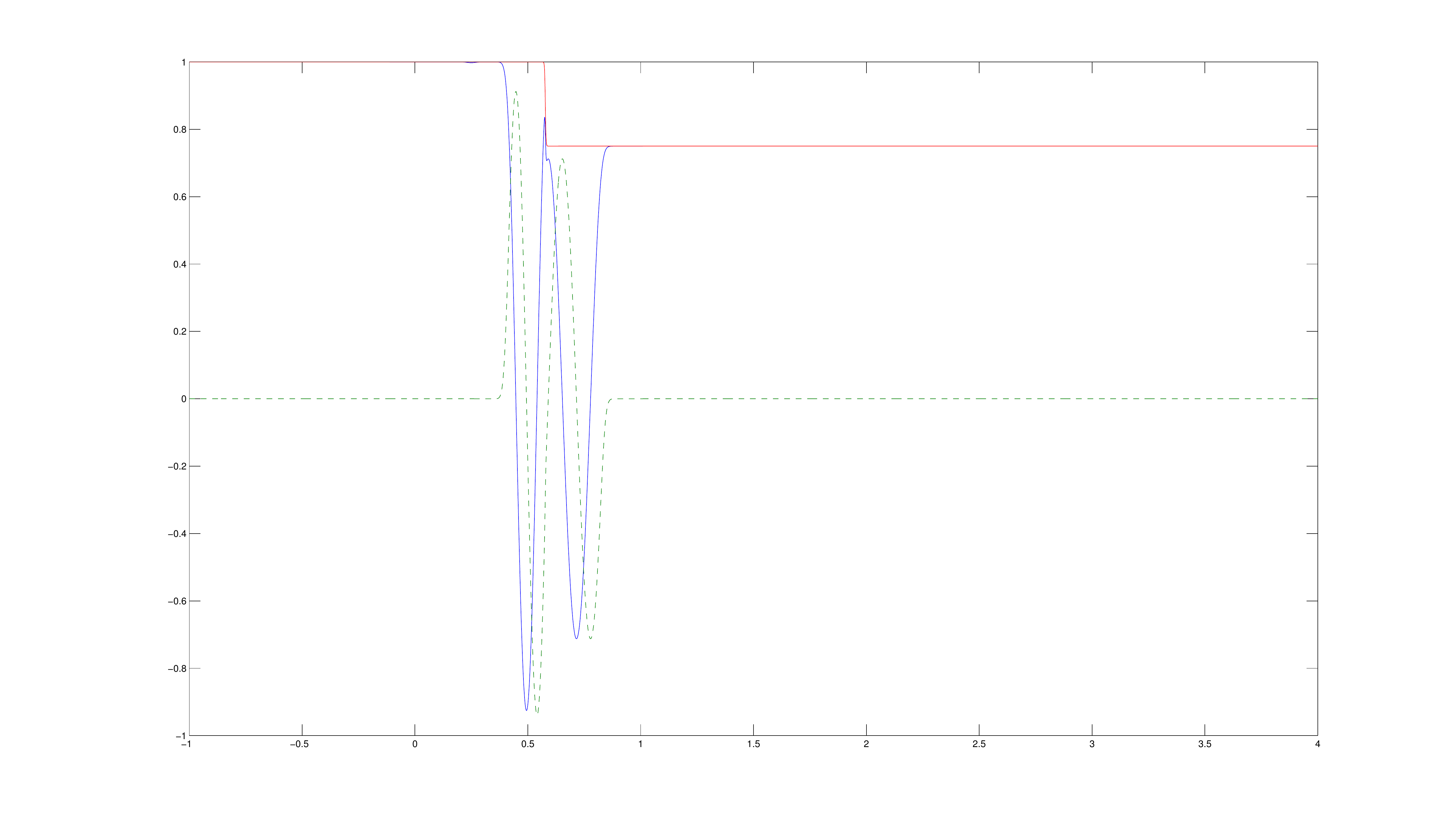}}
  \subfigure[$T =0.25$]{\includegraphics[height=0.4\linewidth,width=0.49\linewidth]{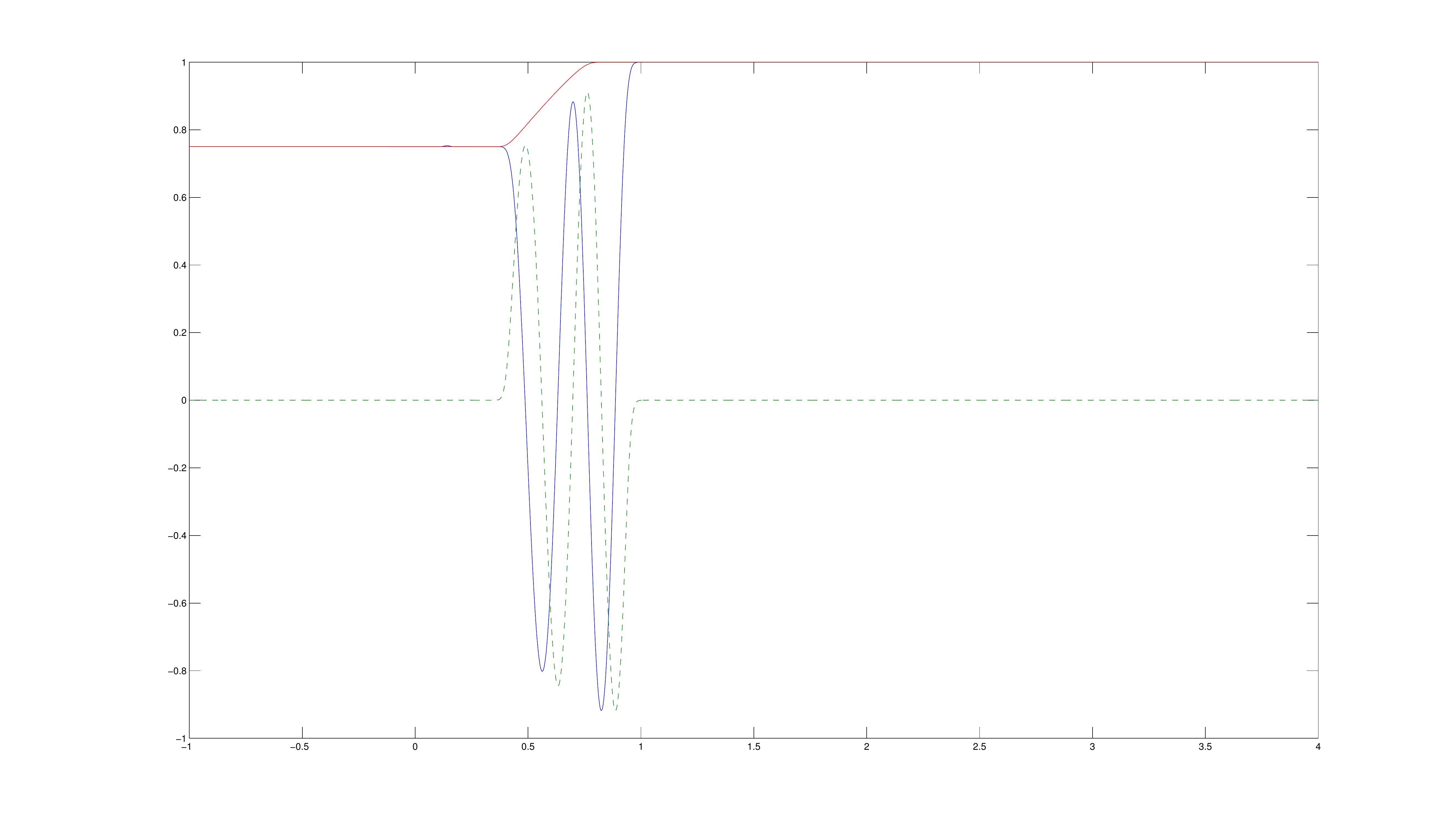}} \\
  \subfigure[$T
  =0.75$]{\includegraphics[height=0.4\linewidth,width=0.49\linewidth]{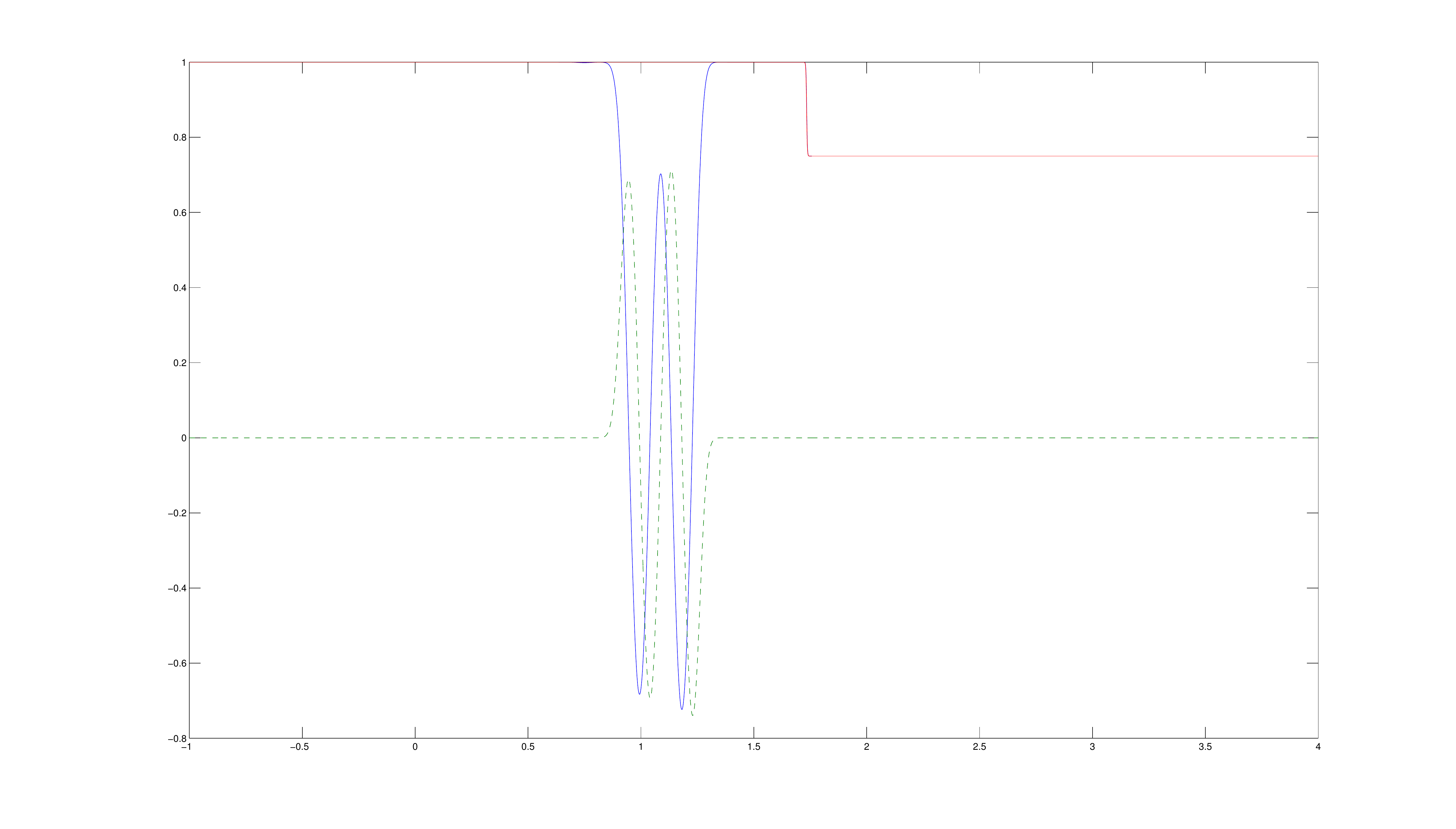}}
  \subfigure[$T
  =0.75$]{\includegraphics[height=0.4\linewidth,width=0.49\linewidth]{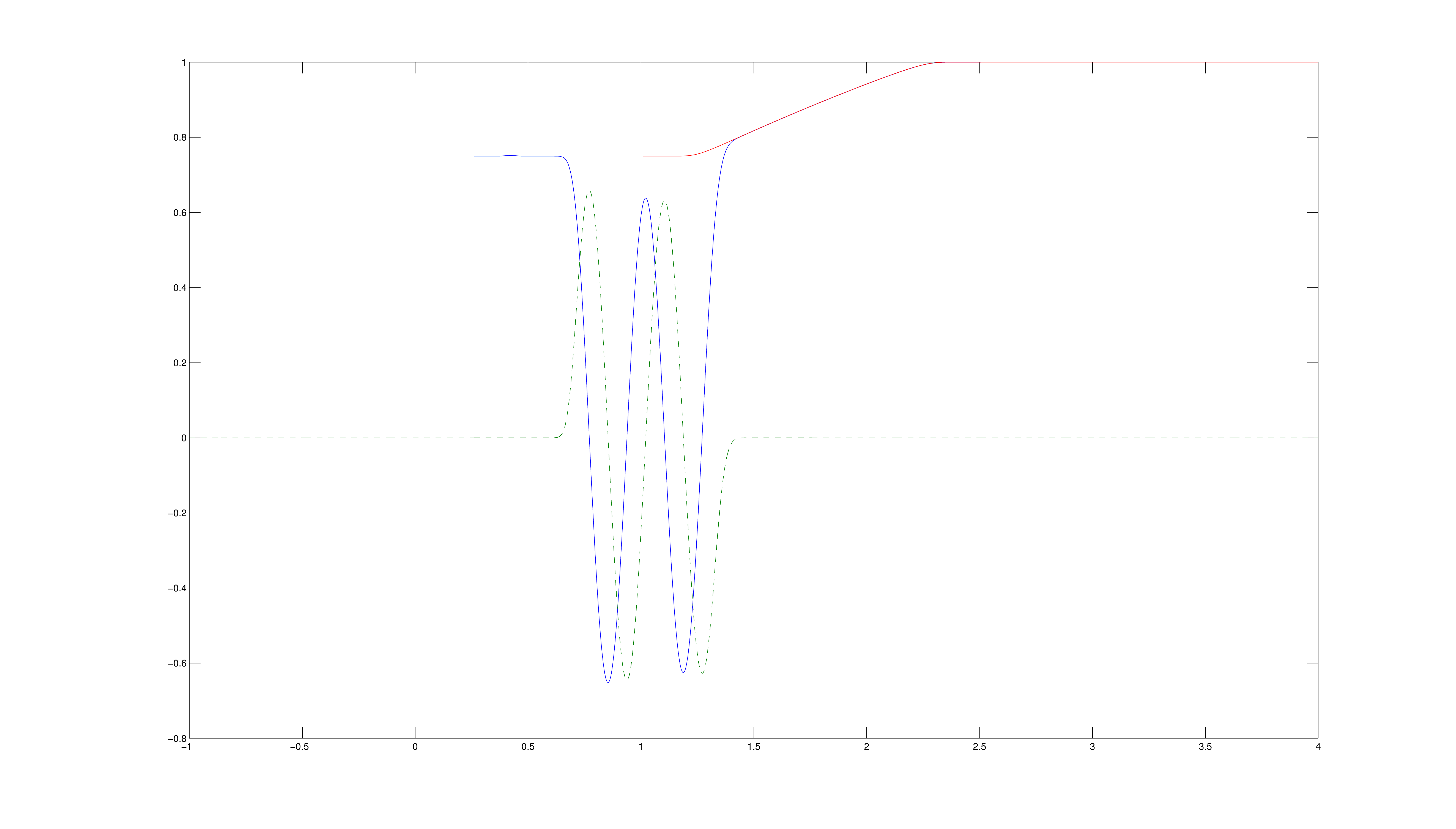}}
  \caption{ Left column: Experiment-$1$: A shock wave initiated behind
    a continuous rotational wave. The dotted-dashed curve represents
    the first component of $U$, the dotted curve represents the second
    component and the solid curve represents $r$. Right column:
    Experiment-$2$: An expansion wave initiated behind a continuous
    rotational wave. The dotted-dashed curve represents the first
    component of $U$, the dotted curve represents the second component
    and the solid curve represents $r$.}
  \protect \label{fig:2}
\end{figure}

\begin{figure}[htbp]
  \centering \subfigure[$T
  =0$]{\includegraphics[height=0.4\linewidth,width=0.49\linewidth]{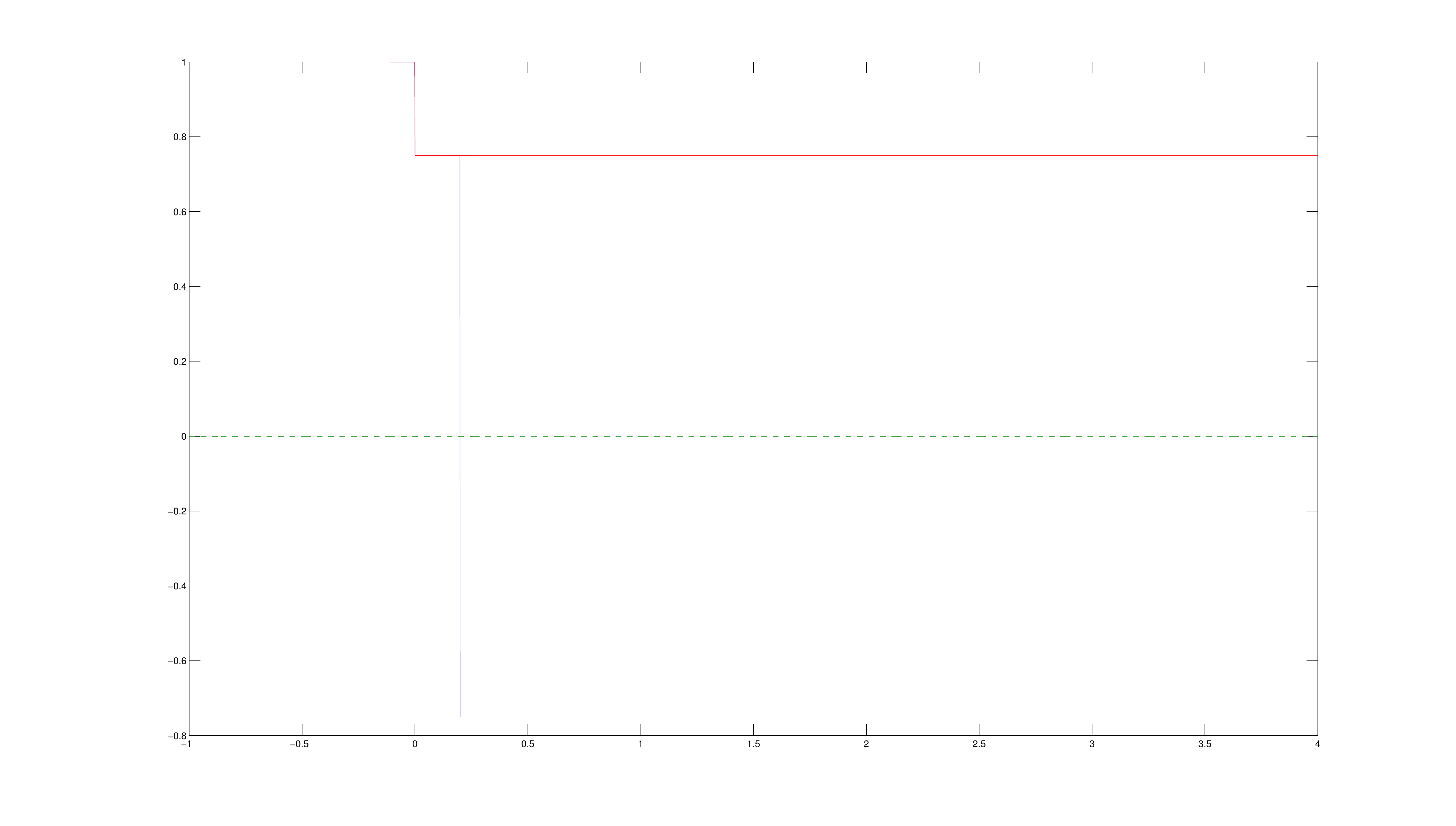}}
  \subfigure[$T =0$]{\includegraphics[height=0.4\linewidth,width=0.49\linewidth]{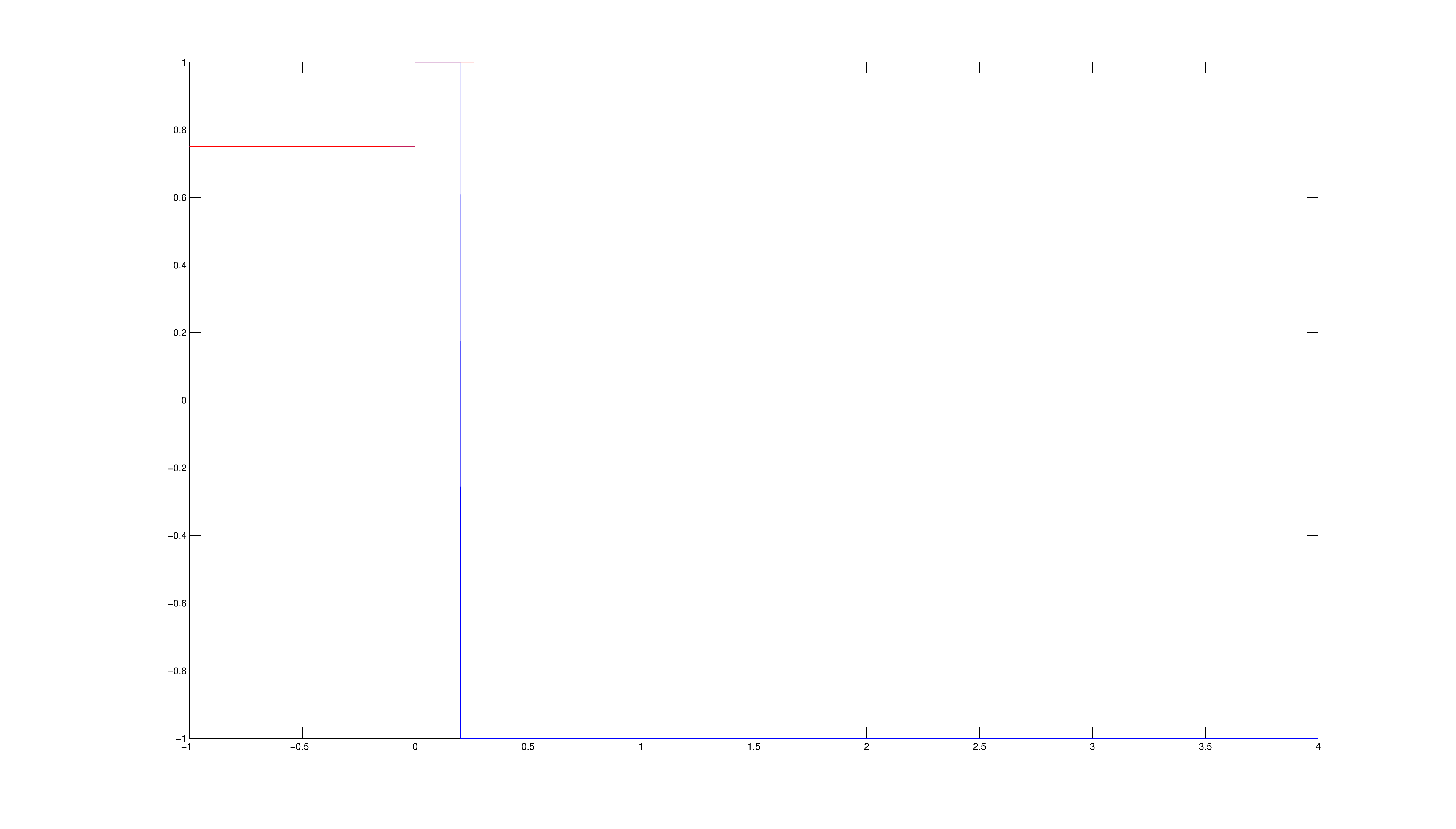}}\\
  \subfigure[$T
  =0.25$]{\includegraphics[height=0.4\linewidth,width=0.49\linewidth]{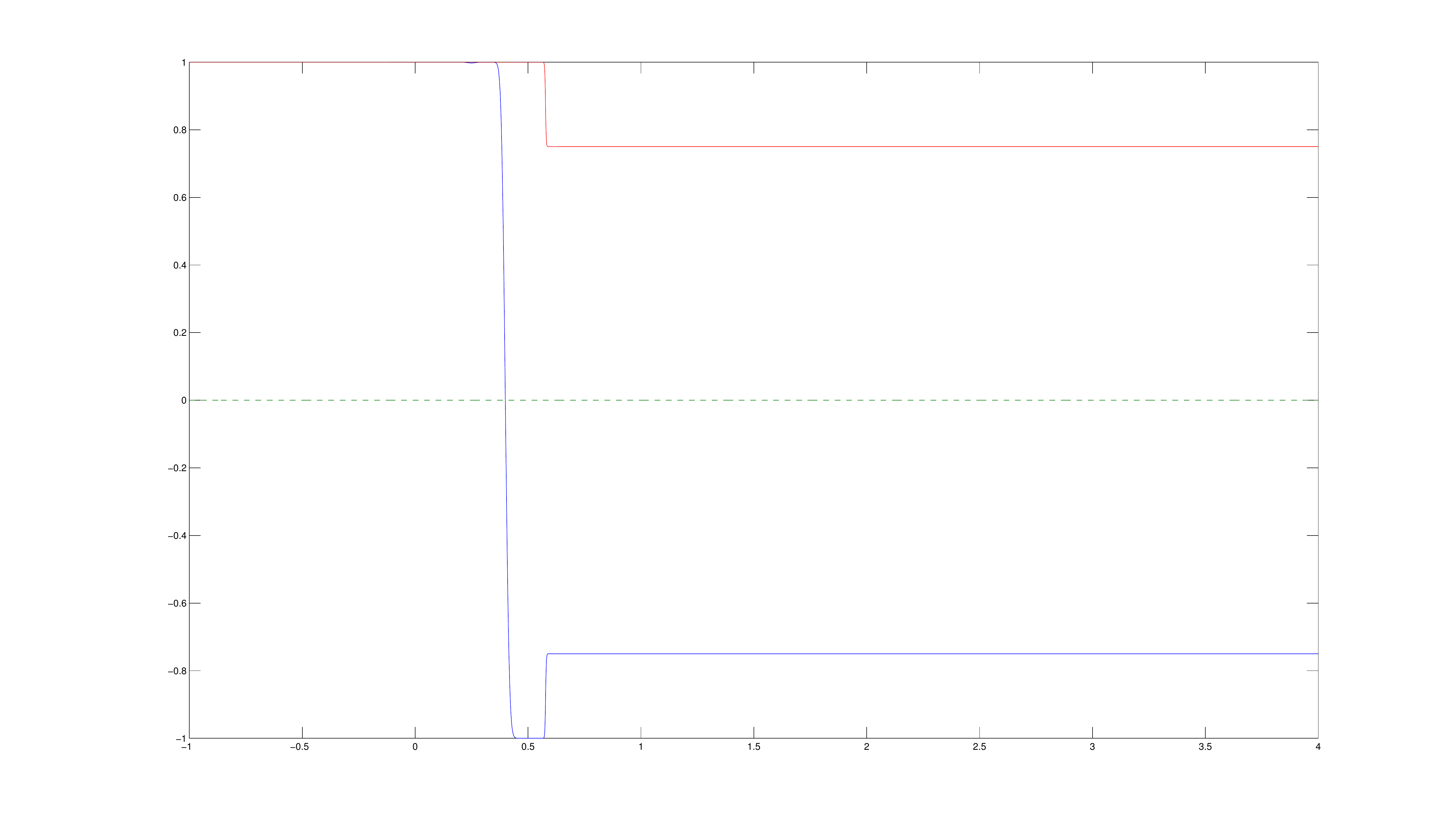}}
  \subfigure[$T =0.25$]{\includegraphics[height=0.4\linewidth,width=0.49\linewidth]{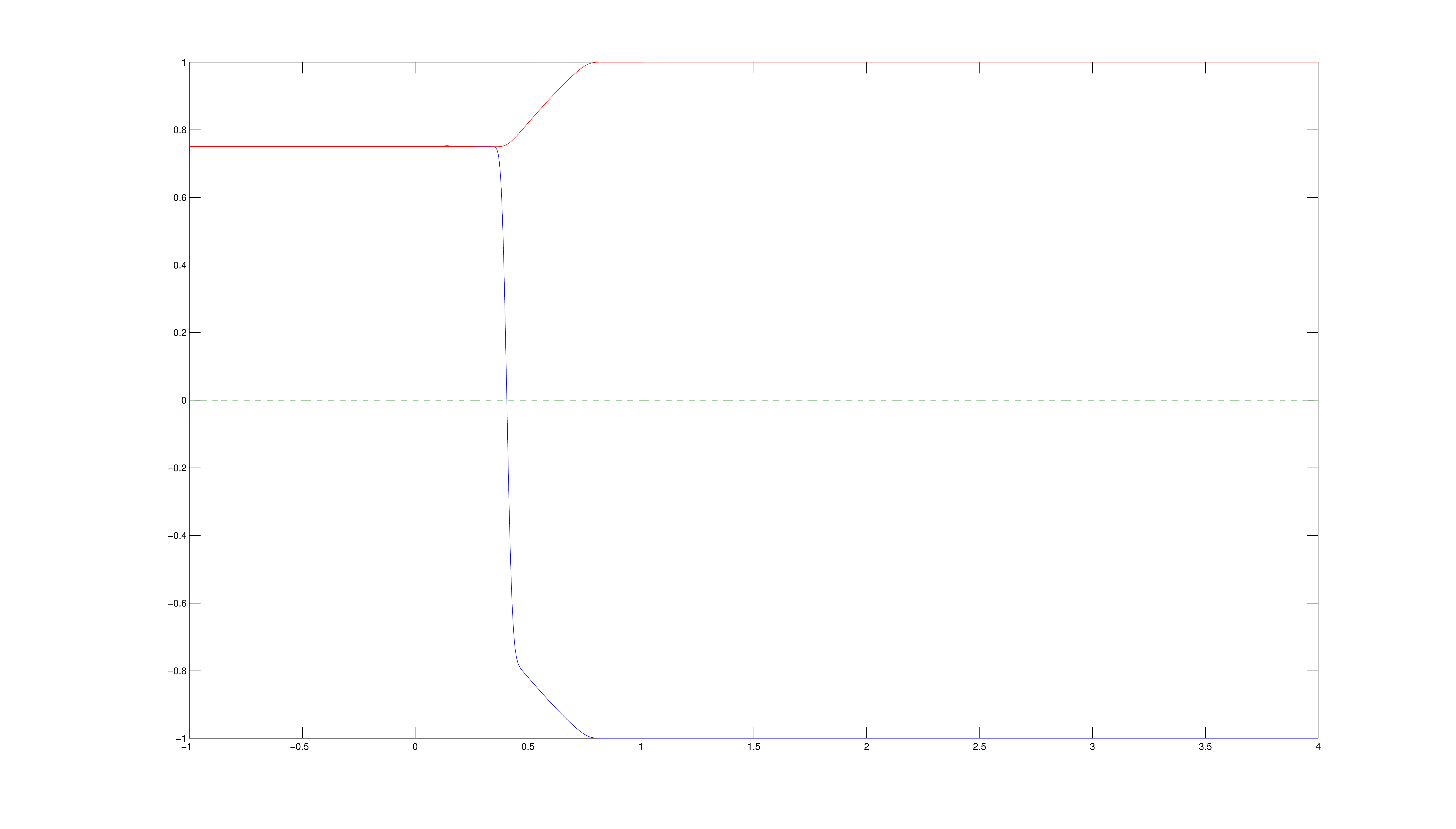}} \\
  \subfigure[$T
  =0.75$]{\includegraphics[height=0.4\linewidth,width=0.49\linewidth]{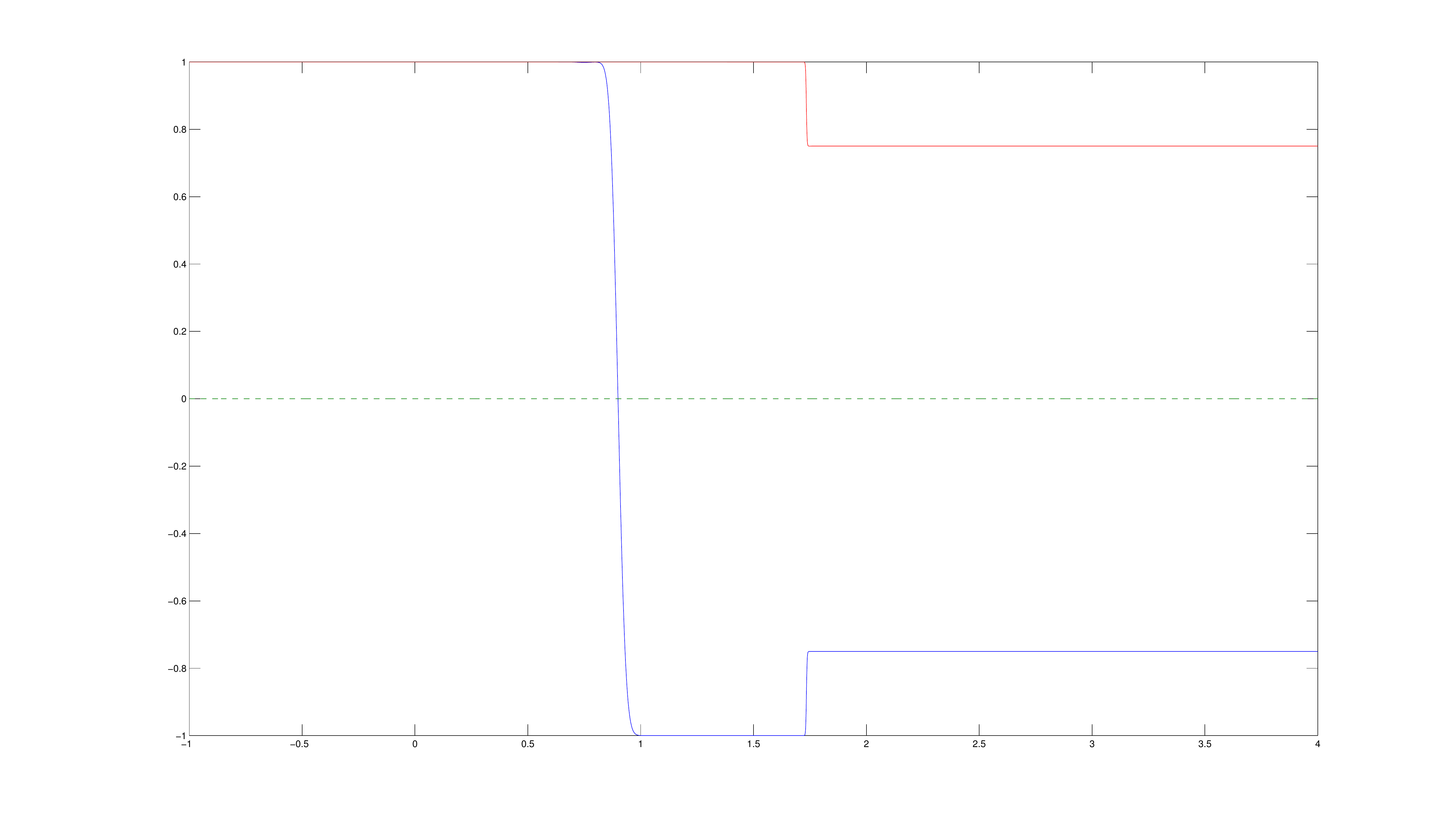}}
  \subfigure[$T
  =0.75$]{\includegraphics[height=0.4\linewidth,width=0.49\linewidth]{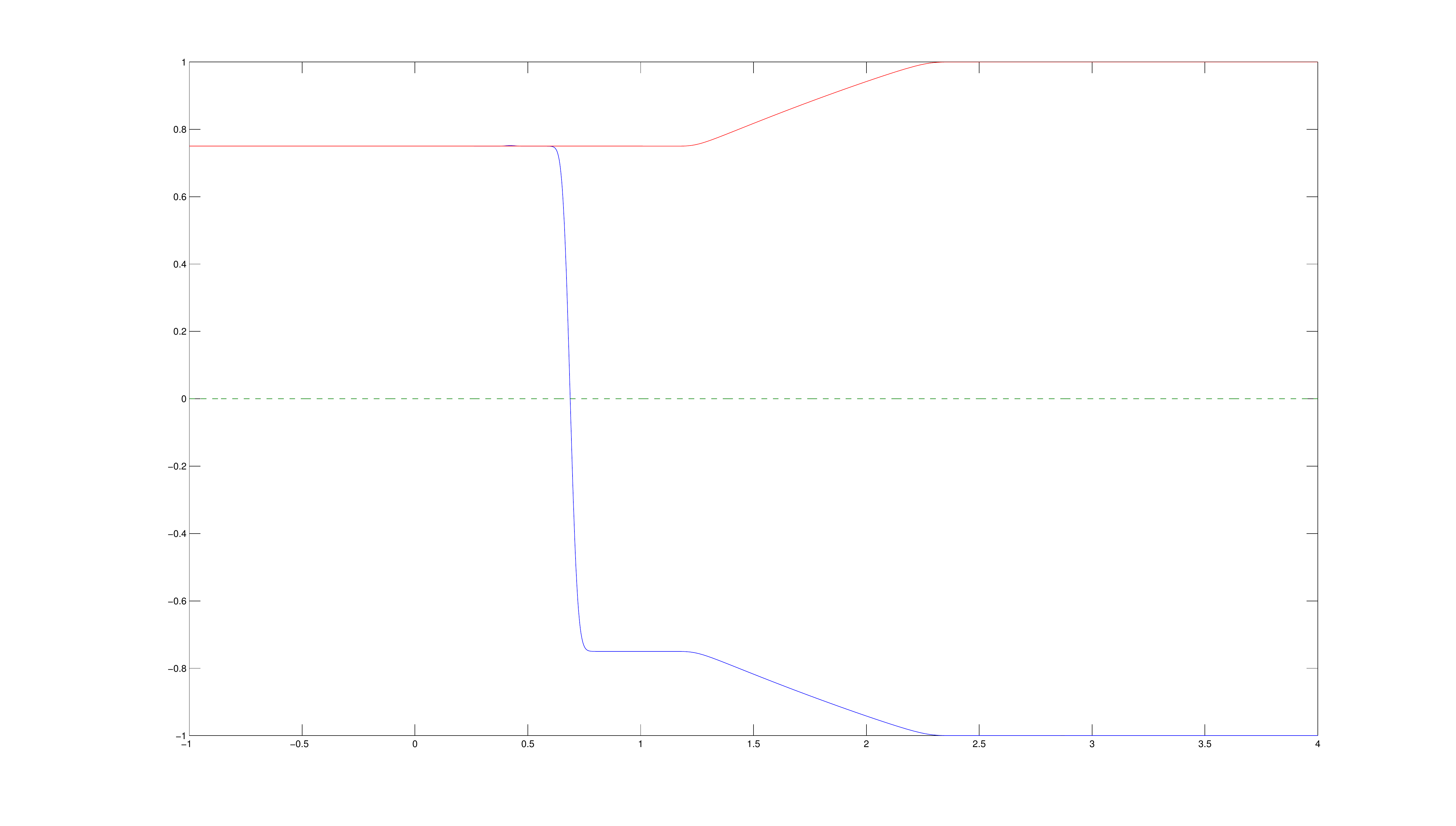}}
  \caption{ Left column: Experiment-$3$: A shock wave initiated behind
    a discontinuous rotational wave. The dotted-dashed curve
    represents the first component of $U$, the dotted curve represents
    the second component and the solid curve represents $r$. Right
    column: Experiment-$4$: An expansion wave initiated behind a
    discontinuous rotational wave. The dotted-dashed curve represents
    the first component of $U$, the dotted curve represents the second
    component and the solid curve represents $r$.}
  \protect \label{fig:3}
\end{figure}

\subsection{Numerical convergence rates}
\label{subsec:numconv}
We have not obtained any theoretical convergence rates for the schemes
presented here. Never the less, it is interesting to check the
possible convergence rate in practice. To this end we have used
Riemann initial data \eqref{eq:numrieinit} with  $U_l=(1,1)$ and
$U_r=(3,1)$. In this case the exact solution is given by formula
\eqref{eq:numriesol1}, i.e., a rarefaction wave followed by a contact
discontinuity.  We define the relative error for a scheme as
\begin{equation}
  \label{eq:errordef}
  E=100\times \frac{\sum_{j} \abs{u^N_j - u(x_j,N\Dt)}}{%
    \sum_j\abs{u(x_j,N\Dt)}},
\end{equation}
where $u$ is the exact solution found by \eqref{eq:numriesol1} and
$u^N_j$ is the approximation computed by the numerical scheme. Note
that this is a first order accurate approximation to the relative
$L^1$ error since $u$ is piecewise continuous. 

We have computed the errors for the three schemes
\eqref{eq:fullydiscrete} (``scheme1''), the conservative scheme
\eqref{eq:rfully_1} -- 
\eqref{eq:ufully} (``scheme2'') and the non-conservative scheme
\eqref{eq:rfully} -- \eqref{eq:vfully}
(``scheme3''). Table~\ref{tab:1} summarizes the results. We computed
the approximations for $t=1$, and used $\Dx=40/2^N$ for
$N=5,\ldots,14$ for $x\in [-1,39]$.
\newcommand{\rb}[1]{\raisebox{1.5ex}{#1}}
\begin{table}[h]
  \centering
  \begin{tabular}[h]{c|r r | r r |r r}
    &\multicolumn{2}{c|}{scheme1}
    &\multicolumn{2}{c|}{scheme2}
    &\multicolumn{2}{c}{scheme3}\\
    $N$ &\multicolumn{1}{c}{$E$} &\multicolumn{1}{c|}{rate}
    &\multicolumn{1}{c}{$E$} &\multicolumn{1}{c|}{rate}
    &\multicolumn{1}{c}{$E$} &\multicolumn{1}{c}{rate} \\
    \hline
    5  &3.32 &   &3.36 &  & 3.40   \\[-1ex]
    6  &2.04 & \rb{0.70} &2.08&\rb{0.68}&2.31&\rb{0.55}\\[-1ex]
    7  &1.31 & \rb{0.64} &1.35&\rb{0.63}&1.50&\rb{0.62}\\[-1ex]
    8  &0.81 & \rb{0.69} &0.83&\rb{0.69}&0.89&\rb{0.75}\\[-1ex]
    9  &0.51 & \rb{0.67} &0.52&\rb{0.69}&0.54&\rb{0.72}\\[-1ex]
    10 &0.32 & \rb{0.68} &0.32&\rb{0.68}&0.33&\rb{0.70}\\[-1ex]
    11 &0.20 & \rb{0.64} &0.20&\rb{0.67}&0.21&\rb{0.68}\\[-1ex]
    12 &0.13 & \rb{0.65} &0.13&\rb{0.67}&0.13&\rb{0.68}\\[-1ex]
    13 &0.09 & \rb{0.62} &0.08&\rb{0.65}&0.08&\rb{0.65}\\[-1ex]
    14 &0.06 & \rb{0.60} &0.05&\rb{0.63}&0.05&\rb{0.63}\\
    \multicolumn{7}{c}{$\hphantom{777}$}\\[-1ex]
  \end{tabular}
    \caption{Relative errors and rates.}
  \label{tab:1}
\end{table}
From this table it emerges that the three schemes produce very similar
errors and numerical convergence rates. This convergence rate is
expected to be not higher than $1/2$, since the solution contains a
contact discontinuity, and the schemes are formally first
order. Table~\ref{tab:1} actually seems to predict that the three
schemes have a convergence rate of about 0.6.


\end{document}